\documentclass[journal,9pt]{IEEEtran}

\usepackage{pifont}
\usepackage{bbding}  

\IEEEoverridecommandlockouts                              

\usepackage{amsmath,amssymb,epsf,epsfig,times}
\usepackage{amsfonts}
\usepackage[all]{xy}
\usepackage{color}
\usepackage{subfigure}
\usepackage{url,cite}
\usepackage{color}
\usepackage[usenames,dvipsnames]{xcolor}
\usepackage[dvipsnames]{xcolor}
\usepackage{amsmath,amssymb,epsf,epsfig,times}
\usepackage[all]{xy}
\usepackage{graphicx,color}
\usepackage{subfigure}
\usepackage{url}
\usepackage{cite}

\newtheorem{theorem}{Theorem}[section]
\newtheorem{lemma}{Lemma}[section]

\newcommand{\norm}[1]{\left\|#1\right\|}

\newcommand{\defeq}{\stackrel{\triangle}{=}}

\newtheorem{definition}[theorem]{Definition}
\newtheorem{assumption}[theorem]{Assumption}

\newtheorem{prop}{Proposition}
\newtheorem{remark}[theorem]{Remark}

\newcommand{\OMIT}[1]{}

\newif\ifpdf
\ifx\pdfoutput\undefined \pdffalse \else \pdfoutput=1 \pdftrue \fi
\ifpdf
\usepackage{graphicx}
\usepackage{epstopdf}
\usepackage{graphicx}
\fi
\title{\LARGE \bf Distributed Convex Optimization for Continuous-Time Dynamics with Time-Varying Cost Functions}

\author{Salar Rahili, Student Member, IEEE, Wei Ren, Fellow, IEEE
\thanks{Salar Rahili and Wei Ren are with the Department of Electrical and Computer Engineering, University of California, Riverside, CA, 92521, USA.
       {Email: srahi001@ucr.edu, ren@ee.ucr.edu.} This work was supported by NSF under grants CMMI-1537729 and ECCS-1611423.}
}
\begin{document}
%

\maketitle

\begin{abstract}
In this paper, a time-varying distributed convex optimization problem is studied for continuous-time multi-agent systems.
The objective is to minimize the sum of local time-varying cost functions, each of which is known to only an individual agent, through local interaction. Here the optimal point is time varying and creates an optimal trajectory.
Control algorithms are designed for the cases of single-integrator and double-integrator dynamics. In both cases, a centralized approach is first introduced to solve the optimization problem.
Then this problem is solved in a distributed manner and a discontinuous algorithm based on the signum function is proposed in each case. In the case of single-integrator (respectively, double-integrator) dynamics, each agent relies only on its own position and the relative positions (respectively, positions and velocities) between itself and its neighbors. A gain adaption scheme is introduced in both algorithms to eliminate certain global information requirement.
To relax the restricted assumption imposed on feasible cost functions, an estimator based algorithm using the signum function is proposed, where each agent uses dynamic average tracking as a tool to estimate the centralized control input.
As a trade-off, the estimator based algorithm necessitates communication between neighbors.
Then in the case of double-integrator dynamics, the proposed algorithms are further extended.
Two continuous algorithms based on, respectively, a time-varying and a fixed boundary layer are proposed as continuous approximations of the signum function. To account for inter-agent collision for physical agents, a distributed convex optimization problem with swarm tracking behavior is introduced for both single-integrator and double-integrator dynamics. It is shown that the center of the agents tracks the optimal trajectory, the connectivity of the agents is maintained and inter-agent collision is avoided. Finally, numerical examples are included for illustration.
\end{abstract}


\IEEEpeerreviewmaketitle

\section{Introduction} \label{sec:introduction}
The distributed optimization problem has attracted a significant attention recently. It arises in many applications of multi-agent systems, where agents cooperate in order to accomplish various tasks as a team in a distributed and optimal fashion. We are interested in a class of distributed convex optimization problems, where the goal is to minimize the sum of local cost functions, each of which is known to only an individual agent.

The incremental subgradient algorithm is introduced as one of the earlier approaches addressing this problem \cite{D12, nowakJ}. In this algorithm an estimate of the optimal point is passed through the network while each agent makes a small adjustment on it. Recently some significant results based on the combination of consensus and subgradient algorithms have been published \cite{D8-11, D3-08, D4-10}. For example, this combination is used in \cite{D3-08} for solving the coupled optimization problems with a fixed undirected graph. A projected subgradient algorithm is proposed in \cite{D4-10}, where each agent is required to lie in its own convex set. It is shown that all agents can reach an optimal point in the intersection of all agents' convex sets even for a time-varying communication graph with doubly stochastic edge weight matrices.

However, all the aforementioned works are based on discrete-time algorithms. Recently, some new research is conducted on distributed optimization problems for multi-agent systems with continuous-time dynamics. Such a scheme has applications in motion coordination of multi-agent systems. For example, multiple physical vehicles modelled by continuous-time dynamics might need to rendezvous at a team optimal location. In \cite{tangJ}, a generalized class of zero-gradient sum controllers is introduced for twice differentiable strongly convex functions under an undirected graph.
In \cite{C4-13}, a continuous-time version of \cite{D4-10} for directed and undirected graphs is studied, where it is assumed that each agent is aware of the convex optimal solution set of its own cost function and the intersection of all these sets is nonempty. Article \cite{C7-12} derives an explicit expression for the convergence rate and ultimate error bounds of a continuous-time distributed optimization algorithm.
In \cite{C6-11}, a general approach is given to address the problem of distributed convex optimization with equality and inequality constraints. 
A proportional-integral algorithm is introduced in \cite{C5, C2-12, C2-14}, where \cite{C2-12} considers strongly connected weight balanced directed graphs and \cite{C2-14} extends these results using discrete-time communication updates. A distributed optimization problem is studied in \cite{C1} with the adaptivity and finite-time convergence properties.

In continuous-time optimization problems, the agents are usually assumed to have single-integrator dynamics. However, a broad class of vehicles requires double-integrator dynamic models. In addition, having time-invariant cost functions is a common assumption in the literature. However, in many applications the local cost functions are time varying, reflecting the fact that the optimal point could be changing over time and creates a trajectory. 
There are just a few works in the literature addressing the distributed optimization problem with time-varying cost functions \cite{Cheukuri,stochopt,ADMM}. In those works, there exist bounded errors converging to the optimal trajectory. For example, the economic dispatch problem for a network of power generating units is studied in \cite{Cheukuri}, where it is proved that the algorithm is robust to slowly time-varying loads. In particular, it is shown that for time-varying loads with bounded first and second derivatives the optimization error will remain bounded. In \cite{stochopt}, a distributed time-varying stochastic optimization problem is considered, where it is assumed that the cost functions are strongly convex, with
Lipschitz continuous gradients. It is proved that under the persistent excitation assumption, a bounded error in expectation will be achieved asymptotically.
In \cite{ADMM}, a distributed discrete-time algorithm based on the alternating direction method of multipliers (ADMM) is introduced to optimize a
time-varying cost function. It is proved that for
strongly convex cost functions with Lipschitz
continuous gradients, if the primal optimal solutions drift
slowly enough with time, the primal and dual variables are close to
their optimal values.

Furthermore, in all articles on distributed optimization mentioned above, the agents will eventually approach a common optimal point while in some applications it is desirable to achieve swarm behavior.
The goal of flocking or swarming with a leader is that a group
of agents tracks a leader with only local interaction while
maintaining connectivity and avoiding inter-agent collision \cite{Cucker07, olfati06, SuTAC09, caoren12}. 
Swarm tracking algorithms are studied in \cite{olfati06} and \cite{SuTAC09}, where it is assumed
that the leader is a neighbor
of all followers and has a constant and time-varying velocity, respectively. In \cite{caoren12}, swarm tracking algorithms via a variable structure approach are
introduced, where the leader is a neighbor of only a subset of
the followers. In the aforementioned studies, the leader plans
the trajectory for the team and the agents are not directly assigned to complete a task cooperatively. In \cite{TuSayed}, the agents are assigned a task to estimate a stationary field while exhibiting cohesive motions. Although optimizing a certain team criterion while performing the swarm behavior is a highly motivated task in many multi-agent applications, it has not been addressed in the literature.

The introduced framework, distributed continuous-time time-varying optimization, is of great significance in
motion coordination. Here, multiple agents cooperatively achieve motion coordination while optimizing a time-varying team objective function with only local information and interaction. For example, multiple spacecraft might need to dock at a moving location distributively with only local
information and interaction such that the total team performance is optimized. Multiple agents moving in a formation or swarm with local information and interaction might
need to cooperatively figure out what optimal trajectory the virtual leader or center of the team should follow
and that knowledge would help the individual agents specify their motions. Furthermore, there is a significant need to use distributed optimization in various applications such as economic dispatch, internet congestion control, and home automation with smart electrical devices. While the studies in the aforementioned applications would be simplified by assuming that the changing rate of the cost functions or the constraints, is small and hence treated as invariant in each time interval, it might be more realistic and relevant to explicitly take into account the time-varying nature of the cost functions or constraints. As a result, distributed continuous-time optimization algorithms with time-varying cost functions or constraints might serve as continuous-time solvers to figure out the optimal trajectory in these applications.

In this paper, we are faced with several challenges such as: 1) Having time-varying cost functions, which generally changes the problem from finding the fixed optimal point to tracking the optimal trajectory. 2) Solving the problem in a distributed manner using only local information and local interaction. 3) Solving the problem for continuous-time single-integrator and double-integrator dynamics, where in the latter case there is only direct control on agents' accelerations. 4)  In our algorithms, the signum function is employed to compensate for the effect of the inconsistent internal time-varying optimization signals among the agents so that the agents can reach consensus. As the signum function might cause chattering in some applications, it is replaced with continuous approximations in some algorithms but additional challenges in analysis would result from the replacement. 5) Providing analysis on optimization error bounds in scenarios where the agents' states cannot reach consensus. 6) The coexistence of the optimization objective and the inherent nonlinearity of the swarm tracking behavior. Our preliminary attempts for solving the distributed convex optimization problem with time-varying cost functions have been presented in \cite{ACC15salar, CDC15salar}.

The remainder of this paper is organized as follows: In Section \ref{sec:notation}, the notation and preliminaries used throughout this paper are introduced. In Section \ref{sec:single}, the case of single-integrator dynamics is studied. In Subsection \ref{seccent}, a centralized approach is introduced. Then, in Subsections \ref{subsecsingledis1} and \ref{secsingledis2}, two discontinuous algorithms are proposed to solve the problem in a distributed manner. In Section \ref{sec:double}, the case of double-integrator dynamics is studied. In Subsection \ref{subsecdoublecent}, a centralized algorithm is introduced. Then in Subsections \ref{subsecdbldis1} and  \ref{subsecdblestim} two discontinuous algorithms are defined to solve the problem in a distributed manner. Subsections \ref{subsecdblsgntimevar} and \ref{subsecdblsgninv} are devoted to extend the proposed discontinuous control algorithms. In the discontinuous algorithms, the signum function is used but it might cause chattering in some applications. Two continuous algorithms are proposed to avoid the chattering effect, where a time-varying and a time-invariant approximation of the signum function are employed in Subsections \ref{subsecdblsgntimevar} and \ref{subsecdblsgninv}, respectively. In Section \ref{sec:swarm}, the distributed convex optimization problem with swarm tracking behavior is studied, where two algorithms for single-integrator and double-integrator dynamics are designed in Subsections \ref{subsecswarmsingle} and \ref{subsecswarm}, respectively. Finally in Section \ref{sec:sim}, numerical examples are given for illustration.

\section{notations and preliminaries} \label{sec:notation}
The following notations are adopted throughout this paper. $\mathbb{R}^+$ denotes the set of positive real numbers. The cardinality of a set $S$ is denoted by $|S|$. ${\mathcal I}$ denotes the index set $\{1,...,N\}$; The transpose of matrix $A$ and vector $x$ are shown as $A^T$ and $x^T$, respectively. $\norm{x}_p$ denotes the p-norm of the vector $x$. We define $\text{sig}(z)^{\alpha}=|z|^{\alpha} \text{sgn}(z),$ where $z \in \mathbb{R}$ and $\alpha>0$. Let $\textbf{1}_n$ and $\textbf{0}_n$ denote the column vectors of $n$ ones and zeros, respectively. $I_n$ denotes the $n \times n$ identity matrix. For matrix $A$ and $B$, the Kronecker product is denoted by $A \otimes B$. The gradient and Hessian of function $f$ are denoted by $\nabla f$ and $H$, respectively. The matrix inequality $A>(\geq) B$, $A<(\leq) B$, $A> (\geq) 0$ and $A< (\leq) 0$ mean that $A-B$, $B-A$, $A$ and $-A$ are positive (semi)definite, respectively.  Let $\lambda_\text{min}[A]$ and $\lambda_\text{max}[A]$ denote, respectively, the smallest and the largest eigenvalue of the matrix $A$.

Let a triplet ${\mathcal G}=({\mathcal V},{\mathcal E},{\mathcal A})$ be an undirected graph, where ${\mathcal V}=\{1,...,N\}$ is the node set and ${\mathcal E} \subseteq {\mathcal V} \times {\mathcal V}$ is the edge set, and ${\mathcal A}=[a_{ij}] \in \mathbb{R}^{N\times N}$ is the adjacency matrix. An edge between agents $i$ and $j$, denoted by $e=(i,j) \in {\mathcal E}$, means that they can obtain information from each other.
In an undirected graph the edges $(i,j)$ and $(j,i)$ are equivalent. We assume $(i,i) \not\in {\mathcal E}$. The  adjacency matrix ${\mathcal A}$ is defined as $a_{ij}=a_{ji}=1$ if $(i,j) \in {\mathcal E}$ and $a_{ij}=0$ otherwise. 
The set of neighbors of agent $i$ is denoted by $N_i=\{j \in {\mathcal V}: (j,i) \in {\mathcal E} \}$.
A sequence of edges of the form $(i,j),(j,k),...,$ where $i,j,k \in {\mathcal V},$ is a path. The graph ${\mathcal G}$ is connected if there is a path from every node to every other node.
By arbitrarily assigning an orientation for the edges in ${\mathcal G}$, let $D = [d_{ik}] \in  \mathbb{R}^{N \times |{\mathcal E}|}$ be the \textit{incidence matrix} associated with ${\mathcal G}$, where $d_{ik} = -1$ if the edge $e_k$ leaves node $i$, $d_{ik} = 1$ if it enters node $i$, and $d_{ik} = 0$ otherwise.
Let the Laplacian matrix $L=[l_{ij}] \in \mathbb{R}^{N \times N}$ associated with the graph ${\mathcal G}$ be defined as $l_{ii} =\sum_{j=1,j \neq i}^N a_{ij}$ and $l_{ij}=-a_{ij}$ for $i \neq j$. Note that $L \triangleq D D^T$. The Laplacian matrix $L$ is symmetric positive semidefinite. The undirected graph ${\mathcal G}$ is connected if and only if $L$ has a simple zero eigenvalue with the corresponding eigenvector $\textbf{1}_N$ and all other eigenvalues are positive \cite{graphtheory}. When the graph ${\mathcal G}$ is connected, we order the eigenvalues of $L$ as $\lambda_1[L]=0 < \lambda_2[L] \leq ...\leq \lambda_N[L]$. Particularly, $\lambda_2[L]$ is the second smallest eigenvalue of the Laplacian matrix $L$.
The above notations can also be adopted for time-varying graphs, where ${\mathcal G} (t), {\mathcal A}(t), D(t)$ and $L(t)$ are, respectively, the undirected graph, the  adjacency matrix, the incidence matrix and the Laplacian matrix at time $t$. For the time-varying graph $\mathcal{G}(t)$, $\lambda_i[L(t)], \forall i \in \mathcal{I},$ is a function of $t$. As long as $\mathcal{G}(t)$ is connected, $\lambda_2[L(t)]$ is uniformly lower bounded above $0$ because there is only a finite number of possible $L(t)$ associated with $\mathcal{G}(t)$.

\begin{lemma} \cite{eig1} \label{eigvalue}
The second smallest eigenvalue $\lambda_2[L]$ of the Laplacian matrix $L$ associated with the undirected connected graph ${\mathcal G}$ satisfies $\lambda_2[L] = \min_{x^T \textbf{1}_N =0, x \neq \textbf{0}_N} \frac{x^T L x}{x^T x}$.
\end{lemma}
\begin{lemma}  \label{lem-intro-gradzero}
Let $f(x):  \mathbb{R}^{m} \rightarrow  \mathbb{R}$ be a continuously differentiable convex function. The function $f(x)$ is minimized at $x^*$ if and only if $\nabla f(x^*)=0$ \cite{bazaraa}. Furthermore, for any strictly convex function $h(x):  \mathbb{R}^{m} \rightarrow  \mathbb{R}$, the optimal solution $x^*$, assuming that it exists, is unique \cite{boyd04}. 
\end{lemma}
\begin{lemma} \cite{boyd} \label{schur}
The symmetric real matrix $\left( {\begin{array}{cc} Q & S \\ S^T & R \\ \end{array} } \right)$ is positive definite if and only if one of the following conditions hold:
(i) $Q> 0, R - S^T Q^{-1} S > 0$; or (ii) $R >0, Q - S R^{-1} S^T > 0.$
\end{lemma}

\section{Time-Varying Convex Optimization For Single-Integrator Dynamics} \label{sec:single}

Consider a multi-agent system consisting of $N$ physical agents with an interaction topology described by the undirected graph ${\mathcal G}$. It is common to adopt single-integrator or double-integrator models. Here, suppose that the agents satisfy the continuous-time single-integrator dynamics
\begin{equation} \label{single}
 \dot{x}_i(t) =u_i(t), \quad i \in \mathcal{I},
\end{equation}
where $x_i(t) \in  \mathbb{R}^{m}$ is the position, and $u_i(t) \in  \mathbb{R}^{m}$ is the control input of agent $i$. Note that $x_i(t)$ and $u_i(t)$ are functions of time. Later for ease of notation we will write them as $x_i$ and $u_i$. A time-varying local cost function $f_i :  \mathbb{R}^{m}\times  \mathbb{R}^{+} \rightarrow  \mathbb{R}$ is assigned to agent $i\in {\mathcal I},$ which is known to only agent $i$. The team cost function is denoted by $\sum_{i=1}^N f_i(x,t)$ and assumed to be convex. Note that here only $\sum_{i=1}^N f_i(x,t)$ is required to be convex but not necessarily each $f_i(x,t)$. Our objective is to design $u_i$ for \eqref{single} using only local information and local interaction with neighbors such that all agents track the optimal state $x^*(t)$, where $x^*(t)$ is the minimizer of the time-varying convex optimization problem
\begin{equation} \label{xstar1}
\min_{x \in  \mathbb{R}^m} \sum_{i=1}^N f_i(x,t).
\end{equation} 
 \begin{assumption} \label{as-exist}
There exists a continuous $x^*(t)$ that minimizes the team cost function $\sum_{i=1}^N f_i(x,t)$.
\end{assumption}

Because the inverse of the Hessian will be used in our algorithm, we need one of the following assumptions to guarantee its existence.
\begin{assumption} \label{as1}
The function $\sum_{i=1}^N f_i(x,t)$ is twice continuously differentiable with respect to $x,$ with invertible Hessian $\sum_{j=1}^N H_j(x_,t),\ \forall x,t$.
\end{assumption}
\begin{assumption} \label{as1each}
Each function $f_i(x,t)$ is twice continuously differentiable with respect to $x,$ with invertible Hessian $H_i(x,t),\ \forall x,t$.
\end{assumption}
\subsection{Centralized Time-Varying Convex Optimization} \label{seccent}

As a first step in this subsection, we focus on the time-varying convex optimization problem of 
\begin{equation} \label{fcent}
\min_{x} f_0(x,t),
\end{equation}
where $f_0 :  \mathbb{R}^{m}\times  \mathbb{R}^{+} \rightarrow  \mathbb{R}$ is convex in $x$, for single-integrator dynamics \begin{equation} \label{singlemodelcent}
 \dot{x} =u,
\end{equation}
where $x, u \in  \mathbb{R}^{m}$ are the system's state and control input, respectively. Next, an algorithm adapted from \cite{time-varying} will be proposed to solve the problem defined by (\ref{fcent}) for the system \eqref{singlemodelcent}.  The control input is proposed for \eqref{singlemodelcent} as
\begin{equation} \label{ucent} 
u= -H_0^{-1}(x,t)(\tau \nabla f_0(x,t)+ \frac{\partial }{\partial t}\nabla f_0(x,t)), 
\end{equation}
where $\tau >0$ is a positive coefficient; $ \nabla f_0(x,t)$ and $H_0(x,t)$ are respectively, the first and the second  derivative of the cost function $f_0(x,t)$ with respect to $x$, namely, the gradient and Hessian.

\begin{theorem} \label{Theoremsinglecentral}
Suppose that $f_0$ satisfies Assumptions \ref{as-exist} and \ref{as1}. Using \eqref{ucent} for \eqref{singlemodelcent}, $x(t)$ converges to the optimal trajectory $x_0^*(t)$, the minimizer of \eqref{fcent}, i.e., $\lim_{t \to \infty} [x(t)-x_0^*(t)] = 0$.
\end{theorem}\normalsize\
\begin{proof}
Define the positive-definite Lyapunov function candidate $W=\frac{1}{2} \nabla f_0(x,t)^T \nabla f_0(x,t)$. The derivative of $W$ along the system \eqref{singlemodelcent} with the control input \eqref{ucent} is $\dot{W}= \nabla f_0(x,t)^T H_0(x,t) \dot{x} +\nabla f_0(x,t)^T \frac{\partial}{\partial t} \nabla  f_0(x,t) = -\tau \nabla f_0(x,t)^T \nabla f_0(x,t)$.
Therefore, $\dot{W}<0$ for $\nabla f_0 \neq 0$. This guarantees that $\nabla f_0$ will asymptotically converge to zero when $t \to \infty$. Then by using Lemma \ref{lem-intro-gradzero} and under Assumption \ref{as-exist}, it is easy to see that $x(t)$ converges to $x_0^*(t)$, and $f_0$ will be minimized.        
\end{proof}
\begin{remark}
There exist other choices for the control input $u$ instead of the one proposed in (\ref{ucent}). For example, $u=-\tau \nabla f_0(x,t)-H_0^{-1}(x,t)\frac{\partial }{\partial t}\nabla  f_0(x,t)$ might be used. In this alternative control input, it can be seen that for a time-invariant cost function, $\frac{\partial}{\partial t} \nabla  f_0(x,t) =0$. Hence we will have the well-known gradient descent algorithm. For a time-invariant cost function, the proposed algorithm \eqref{ucent} will become a Newton algorithm, which is generally much faster than the gradient descent algorithm.
\end{remark}

The results from Theorem \ref{Theoremsinglecentral} can be extended to minimize the convex function $\sum_{i=1}^N f_i(x,t)$. If Assumptions \ref{as-exist} and \ref{as1} hold, with
\begin{equation}\small \label{sum-min} 
u=({\sum_{j=1}^N H_j(x,t)})^{-1}  \big(\tau \sum_{j=1}^N\nabla f_j(x,t)+ \sum_{j=1}^N\frac{\partial}{\partial t} \nabla  f_j(x,t)\big)
\end{equation} \normalsize
for \eqref{singlemodelcent}, the function $\sum_{i=1}^N f_i(x,t)$ is minimized. Unfortunately, \eqref{sum-min} is a centralized solution for agents with single-integrator dynamics relying on the knowledge of all $f_i, i\in \mathcal{I}$. 
In Subsections \ref{subsecsingledis1} and \ref{secsingledis2}, \eqref{sum-min} will be exploited to propose two algorithms for solving the time-varying convex optimization problem for single-integrator dynamics in a distributed manner.
\subsection{Distributed Time-Varying Convex Optimization Using Neighbors' Positions} \label{subsecsingledis1}

In this subsection, we focus on solving the distributed time-varying convex optimization problem \eqref{xstar1} for agents with single-integrator dynamics \eqref{single}. Each agent has access to only its own position and the relative positions between itself and its neighbors. In some applications, the relative positions can be obtained by using only agents' local sensing capabilities, which might in turn eliminate the communication necessity between agents. The problem defined in \eqref{xstar1} is equivalent to
\begin{equation} \label{costdis}
\min_{x_i} \sum_{i=1}^N f_i(x_i,t) \ \text{subject to} \ x_i=x_j,  \ \  \forall i,j \in {\mathcal I}.
\end{equation} 
Intuitively, the problem is deformed as a consensus problem and a minimization problem on the team cost function $\sum_{i=1}^N f_i(x_i,t)$. Here the goal is that the states $x_i(t),  \forall i \in {\mathcal I},$ converge to the optimal trajectory $x^*(t)$, i.e., \begin{align} \label{goal}
\lim_{t \to \infty} [x_i(t)-x^*(t)] = 0.
\end{align}
The control input is proposed for \eqref{single} as
\begin{align} \label{usingledis1}
u_i=& - \sum_{j \in N_i} \beta_{ij}\text{sgn} \big( x_i -x_j \big) +\phi_i, \notag \\
\dot{\beta}_{ij}=&\norm{x_i -x_j}_1,   \ \  \ j \in N_i,\\
\phi_i \triangleq &-  H^{-1}_i(x_i,t)\big( \nabla f_i(x_i,t)+ \frac{\partial}{\partial t} \nabla f_i(x_i,t)\big),\notag
\end{align}
where $\phi_i$ is an internal signal, $\beta_{ij}$ is a varying gain with $\beta_{ij}(0)=\beta_{ji}(0)\geq 0$, and sgn($\cdot$) is the signum function defined componentwise. Note that $\phi_i$ depends on only agent $i $'s position. Here \eqref{usingledis1} is a discontinuous controller. It is worth mentioning that unlike continuous or smooth systems, the equilibrium concept of setting the right hand equal to zero to find the equilibrium point might not be valid for discontinuous systems.
Let $X=[x_1^T,x_2^T,...,x_N^T]^T$, and $\Phi=[\phi_1^T,\phi_2^T,...,\phi_N^T]^T$ denote, respectively, the aggregated states and the aggregated internal signals of the $N$ agents. We also define $\Pi \triangleq I_N- \frac{1}{N} \mathbf{1}_N \mathbf{1}_N^T$. Define agent $i$'s consensus error as $e_{X_i}=x_i-\frac{1}{N}\sum_{\ell=1}^N x_\ell$.
Define the consensus error vector $e_X=(\Pi \otimes I_m)X$.
Note that $\Pi$ has one simple zero eigenvalue with $\mathbf{1}_N$ as its right eigenvector and has $1$ as its other eigenvalue with the multiplicity $N-1$. Then it is easy to see that $e_X=0$ if and only if $x_i=x_j\ \forall i,j \in {\mathcal I}$.
\begin{remark} \label{fillipov}
With the signum function in the proposed algorithms in this paper, the right-hand sides of the closed-loop systems are discontinuous.
Thus, the solution should be investigated in terms of differential inclusions
by using nonsmooth analysis \cite{cortes08, fillipov}. However, since the signum function is measurable and locally
essentially bounded, the Filippov solutions of the closed-loop dynamics always exist.
Also the Lyapunov
function candidates adopted in the proofs hereafter are continuously differentiable. Therefore, the
set-valued Lie derivative of them is a singleton at the discontinuous points and the proofs still hold. To avoid symbol redundancy, we do not use the differential inclusions in the proofs.
Furthermore, Filippov solutions
are absolutely continuous curves \cite{cortes08}, which means that the agents' states are continuous functions.
\end{remark}

The remainder of this subsection is devoted to the verification of the algorithm (\ref{usingledis1}). In Proposition \ref{propsingledis1}, we will show that the agents reach consensus using (\ref{usingledis1}). Then this result will be used in Theorem \ref{theoremsingledis1} to prove that the agents minimize the team cost function as  $t\rightarrow\infty$.

\begin{definition}\label{def-new-lap}
Defining $a'_{ij}=a_{ij} \beta_{ij}$, a new Laplacian matrix $L'=[l'_{ij}] \in  \mathbb{R}^{N \times N}$ is introduced, where $l'_{ii} =\sum_{j=1,j \neq i}^N a'^2_{ij}$ and $l'_{ij}=-a'^2_{ij}$ for $i\neq j$. Since $a'_{ij}=a'_{ji}$, the matrix $L'$ is symmetric.
	Similar to the definition of $D$, $D' \triangleq [d'_{ij}] \in  \mathbb{R}^{n \times m}$ is the incidence matrix associated with $L'$, where $d'_{ij} = -a'_{ij}$ if the edge $e_j$ leaves node $i$, $d_{ij} = a'_{ij}$ if it enters node $i$, and $d_{ij} = 0$ otherwise. Thus, $L'$ can be given by $L'=D'D'^T$.
\end{definition}			
\begin{assumption}\label{as-bound-single}
With $\phi_i$ defined in  \eqref{usingledis1}, there exists a positive constant $\bar{\phi}$ such that $\norm{\phi_i -\phi_j}_2 \leq \bar{\phi}, \  \forall i,j \in {\mathcal I},$ and $\forall t$.
\end{assumption}
\begin{prop} \label{propsingledis1}
Suppose that the graph ${\mathcal G}$ is connected and Assumption \ref{as-bound-single} holds. The system \eqref{single} with the algorithm (\ref{usingledis1}) reaches consensus, i.e, $x_i=x_j, \forall i,j \in {\mathcal I},$ as $t\rightarrow \infty$.  
\end{prop}

\begin{proof}
Using Definition \ref{def-new-lap}, the closed-loop system \eqref{single} with the control input (\ref{usingledis1}) can be recast into a compact form as
\begin{equation} \label{closesingledis1}
\dot{X} = -  (D' \otimes I_m) \text{sgn} \big([ D^T \otimes I_m]X \big)+\Phi,
\end{equation}
where $D$ and $D'$ are defined in Section \ref{sec:notation} and Definition \ref{def-new-lap}, respectively. We can rewrite (\ref{closesingledis1}) as 
\begin{equation} \small  \label{errorsingle1}
\begin{split}
\dot{e}_X = &-  (D' \otimes I_m) \text{sgn}\big([D^T \otimes I_m]e_X \big)+(\Pi \otimes I_m)\Phi,
\end{split}
\end{equation} \normalsize
where we have used the fact that $\Pi D'=D'$. Define the Lyapunov function candidate 
\begin{align*}
W=\frac{1}{2} e_X^T e_X+\frac{1}{2}\sum_{i=1}^{N}\sum_{j \in N_i} (\beta_{ij}-\bar{\beta})^2,
\end{align*}
where $\bar{\beta}>0$ is to be selected. The time derivative of $W$ along (\ref{errorsingle1}) can be obtained as
\vspace{-.1cm}
\begin{align}\label{sgnnorm} \small
\dot{W}=&- e_X^T (D' \otimes I_m)\text{sgn}\big([D^T \otimes I_m]e_X \big)+e_X^T (\Pi \otimes I_m)\Phi \notag\\
&+\frac{1}{2}\sum_{i=1}^{N}\sum_{j \in N_i} (\beta_{ij}-\bar{\beta})\dot{\beta}_{ij}\notag\\
=&-\sum_{i=1}^{N}\sum_{j \in N_i}\frac{\beta_{ij}}{2}\norm{e_{X_i}-e_{X_j}}_1  \\
&+\frac{1}{2N}\sum_{i=1}^{N}\sum_{j=1}^{N} (e_{X_i}-e_{X_j})(\phi_i-\phi_j)\notag\\
&+\frac{1}{2}\sum_{i=1}^{N}\sum_{j \in N_i} (\beta_{ij}-\bar{\beta})(\norm{x_i -x_j}_1)\notag
\end{align}
\begin{align*}\small
&=-\sum_{i=1}^{N}\sum_{j \in N_i}\frac{\beta_{ij}}{2}\norm{e_{X_i}-e_{X_j}}_1\\
&+\frac{1}{2N}\sum_{i=1}^{N}\sum_{j=1}^{N} (e_{X_i}-e_{X_j})(\phi_i-\phi_j)\notag
-\frac{\bar{\beta}}{2}\sum_{i=1}^{N}\sum_{j \in N_i}\norm{e_{X_i}-e_{X_j}}_1\\
&+\sum_{i=1}^{N}\sum_{j \in N_i} \frac{\beta_{ij}}{2}\norm{e_{X_i}-e_{X_j}}_1 \notag\\
&\leq \frac{1}{2N}\sum_{i=1}^{N}\sum_{j=1}^{N}\norm{e_{X_i}-e_{X_j}}_1\norm{\phi_i-\phi_j}\\
&-\frac{\bar{\beta}}{2}\sum_{i=1}^{N}\sum_{j \in N_i}\norm{e_{X_i}-e_{X_j}}_1\\
&\leq \frac{\bar{\phi}}{2N}\sum_{i=1}^{N}\sum_{j=1}^{N}\norm{e_{X_i}-e_{X_j}}_1-\frac{\bar{\beta}}{2}\sum_{i=1}^{N}\sum_{j \in N_i}\norm{e_{X_i}-e_{X_j}}_1,\notag 
\end{align*}\normalsize
where the last inequality holds under Assumption \ref{as-bound-single}. Because $\mathcal{G}$ is connected, we have
\begin{equation*} \small
\begin{split} 
&\dot{W} \leq \frac{\bar{\phi}}{2} \max_i{\{ \sum_{j=1,j \neq i}^{N}\norm{e_{X_i}-e_{X_j}}_1}\}-\frac{\bar{\beta}}{2}\sum_{i=1}^{N}\sum_{j \in N_i}\norm{e_{X_i}-e_{X_j}}_1\\
& \leq \frac{(N-1)\bar{\phi}}{4}\sum_{i=1}^{N}\sum_{j \in N_i}\norm{e_{X_i}-e_{X_j}}_1-\frac{\bar{\beta}}{2}\sum_{i=1}^{N}\sum_{j \in N_i}\norm{e_{X_i}-e_{X_j}}_1.
\end{split}
\end{equation*}\normalsize
Selecting $\bar{\beta}$ such that $\bar{\beta}  >\frac{ (N-1)\bar{\phi}}{2},$ we have
\begin{align} \label{local-lambda2}
\dot{W}\leq& (\frac{(N-1)\bar{\phi}}{4}-\frac{\bar{\beta}}{2})\sum_{i=1}^{N}\sum_{j \in N_i}\norm{e_{X_i}-e_{X_j}}_1\\ \notag
=&(\frac{(N-1)\bar{\phi}}{2}-\bar{\beta})e_X^T (D \otimes I_m)\text{sgn}\big([D^T \otimes I_m]e_X \big)\\\notag
&\leq(\frac{(N-1)\bar{\phi}}{2}-\bar{\beta}) \norm{(D^T \otimes I_m)e_X}_1\\\notag
& \leq (\frac{(N-1)\bar{\phi}}{2}-\bar{\beta}) \sqrt{e_X^T (DD^T \otimes I_m)e_X}\\\notag
&\leq (\frac{(N-1)\bar{\phi}}{2}-\bar{\beta})\sqrt{\lambda_{2}[L]} \norm{e_X}_2< 0,
\end{align}
where in the last inequality the fact that $L=DD^T$ and Lemma \ref{eigvalue} have been used. Therefore, having $W\geq 0$ and $\dot{W} \leq 0$, we can conclude that $e_X \in {\mathcal L}_{\infty}$. By integrating both sides of \eqref{local-lambda2}, we can see that  $e_X \in \mathcal{L}_{2}$. Now, applying Barbalat's Lemma \cite{sheida22j}, we obtain that $e_X$ will converge to zero asymptomatically and hence the agents' positions reach consensus, i.e, $x_i=x_j, \forall i,j \in {\mathcal I},$ as $t\rightarrow \infty$.
\end{proof}

\begin{theorem} \label{theoremsingledis1}
Suppose that the graph ${\mathcal G}$ is connected, and Assumptions \ref{as-exist}, \ref{as1each} and \ref{as-bound-single} hold. If $H_i(x_i,t) = H_j(x_j,t), \forall  t, \ \forall i,j \in {\mathcal I}$, by employing the algorithm (\ref{usingledis1}) for the system (\ref{single}), the optimization goal \eqref{goal} is achieved.
\end{theorem}
\begin{proof}
Define the Lyapunov function candidate\small
\begin{equation}\label{Wsingledis}  W= \frac{1}{2} ( \sum_{j=1}^N \nabla f_j(x_j))^T (\sum_{j=1}^N \nabla f_j(x_j)), 
\end{equation} \normalsize
where $W$ is positive definite with respect to $\sum_{j=1}^N \nabla f_j(x_j)$. The time derivative of $W$ can be obtained as $\dot{W}= (\sum_{j=1}^N \nabla f_j(x_j))^T (\sum_{j=1}^N H_j \dot{x}_j+ \sum_{j=1}^N \frac{\partial }{\partial t} \nabla f_j(x_j)).$ Under the assumption of identical Hessians, we will have
\begin{equation}  \small\label{local2}
\dot{W}= (\sum_{j=1}^N \nabla f_j(x_j))^T (H_i)(\sum_{j=1}^N  \dot{x}_j+ H_i^{-1}\sum_{j=1}^N \frac{\partial }{\partial t} \nabla f_j(x_j)). 
\end{equation} \normalsize
On the other hand, by using \eqref{usingledis1} for the system \eqref{single} and summing up both sides for $j \in \mathcal{I}$, we know that $\sum_{j=1}^N \dot{x}_j = \sum_{j=1}^N \phi_j$. Then we can rewrite \eqref{local2} as $\dot{W}= -(\sum_{j=1}^N \nabla f_j(x_j))^T (\sum_{j=1}^N \nabla f_j(x_j).$ Therefore, $\dot{W}<0$ for $\sum_{j=1}^N \nabla f_j(x_j) \neq 0$.
This guarantees that $\sum_{j=1}^N \nabla f_j(x_j)$ will asymptomatically converge to zero.
Now, under the assumption that $\sum_{i=1}^N f_i(x,t)$ is convex, using Proposition \ref{propsingledis1} and Lemma \ref{lem-intro-gradzero}, it is easy to see that under Assumption \ref{as-exist} as $t\rightarrow \infty$ the team cost function $\sum_{i=1}^N f_i(x_i,t)$ will be minimized, where $x_i=x_j, \ \forall i,j \in {\mathcal I}$.
\end{proof}

\begin{remark} In \eqref{usingledis1} each agent $i$ is required to know $ \frac{\partial }{\partial t} \nabla f_i(x_i,t),$ which might be restrictive. However, there are applications where each agent knows the closed form of its own local cost function (e.g., motion control with an optimization objective) or at least the agent knows how the cost function is varying with respect to time (e.g., home automation).
For example, in motion control with an optimization objective, it is possible that each agent knows the closed form of its local cost function or in home automation smart electrical devices need to agree on
the total amount of energy consumption that maximizes an
overall utility function formed by the sum of the utility
functions of the devices. However, a varying price rate for electricity during a day makes the optimization problem time varying. Although the price rate of the electricity is varying during the day, it is known to the agents beforehand. Hence, calculating$ \frac{\partial}{\partial t}\nabla f_i(x_i,t)$ might not be an issue in this application.
Furthermore, there are algorithms to estimate the derivative of a function by knowing only the value of the function at each time $t$.
How to apply the idea to distributed continuous-time time-varying optimization is a possible direction for our future studies.
\end{remark}

\begin{remark} \label{Remarkboundsingle}
Assumption \ref{as-bound-single} intuitively places a bound on the Hessians and the changing rates of the gradients of the cost functions with respect to $t$.
In Appendix \ref{AppendixA}, we will show that Assumption \ref{as-bound-single} holds if the cost functions with identical Hessians satisfy certain conditions such that the boundedness of $\norm{x_i-x_j}_2$ for all $t$ guarantees the boundedness of $\norm{\nabla f_j(x_j,t) - \nabla f_i(x_i,t) }_2$ and $\norm{\frac{\partial }{\partial t}\nabla  f_j(x_j,t)-\frac{\partial }{\partial t}\nabla  f_i(x_i,t)}_2, \ \forall i,j \in {\mathcal I},$ for all $t$. For example, consider the cost functions commonly used for energy minimization, e.g., $f_i(x_i,t)=(ax_i+g_i(t))^{2},$ where $a$ is a positive constant and $g_i(t)$ is a time-varying function particularly for agent $i$. For these cost functions, the boundedness of $\norm{x_i-x_j}_2$ for all $t$ guarantees the boundedness of $\norm{\nabla f_j(x_j,t) - \nabla f_i(x_i,t) }_2$ and $\norm{\frac{\partial }{\partial t}\nabla  f_j(x_j,t)-\frac{\partial }{\partial t}\nabla  f_i(x_i,t)}_2, \ \forall i,j \in {\mathcal I},$ for all $t$, if $\norm{g_i(t) -g_j(t)}_2$ and $\norm{\dot{g}_i(t) -\dot{g}_j(t)}_2$ are bounded. Hence to satisfy Assumption \ref{as-bound-single} for $f_i(x_i,t)=(ax_i+g_i(t))^{2},$ it is sufficient to have a bound on $\norm{g_i(t) -g_j(t)}_2$ and $\norm{\dot{g}_i(t) -\dot{g}_j(t)}_2$.
\end{remark}

In Subsection \ref{secsingledis2} an estimator-based algorithm is introduced, where the assumption on identical Hessians is relaxed.

\subsection{Estimator-Based Distributed Time-Varying Convex Optimization} \label{secsingledis2}

In this subsection, an estimator-based algorithm is designed such that each agent calculates (\ref{sum-min}) in a distributed manner. To achieve this goal, distributed average tracking is used as a tool. Each agent generates an estimate of (\ref{sum-min}). Then a controller is designed such that each agent tracks its own generated signal while guaranteeing that the agents reach consensus.

The proposed algorithm for the system (\ref{single}) has two separate parts, the estimator and controller. The estimator part is given by
  \begin{align} \label{avali}\small
\dot{\xi}_i=\alpha \sum_{j \in N_i(t)} \text{sgn}(w_j - w_i), \quad w_i=\xi_i+{\nabla} f_i(x_i,t)
  \end{align}
\vspace{-.45cm}
  \begin{align}\label{dovomi}
\hspace*{-.5cm}\dot{\psi}_i=\beta \sum_{j \in N_i(t)} \text{sgn}(\theta_j -\theta_i), \ \quad \theta_i=\psi_i+H_i(x_i,t)
  \end{align}
\vspace{-.45cm}
  \begin{align}\label{sevomi}
\dot{\phi}_i=\gamma \sum_{j \in N_i(t)} \text{sgn}(\varsigma_j -\varsigma_i), \quad \varsigma_i=\phi_i+\frac{{\partial}}{{\partial t}}\nabla f_i(x_i,t)
  \end{align}
\vspace{-.45cm}
\begin{align} \label{chahar}\hspace{0.4cm}
S_i=&-{\theta_i}^{-1}(\tau w_i+\varsigma_i),
  \end{align}
where  $\alpha, \beta, \gamma$, and $\tau$ are positive coefficients to be selected and $N_i(t)$ is the set of agent $i$'s neighbors at time $t$. The controller part is given by
  \begin{equation}\label{cont}
u_i = -\sum_{j \in N_i(t)} \text{sig}(x_i -x_j)^{\eta}+S_i,
  \end{equation} \normalsize
where $\text{sig}(\cdot)$ is defined componentwise and $0<\eta<1$.
In implementing \eqref{chahar}, $\theta_i$ can be projected on the space of positive-definite matrices, which ensures that $\theta_i$ remains nonsingular. Also $\xi_i,\psi_i,$ and $\phi_i$ are the internal states of the distributed average tracking estimators, where their initial values are such that\footnote{As a special case the initial values can be chosen as $\xi_j(0)=\psi_j(0)=\phi_j(0)=0, \forall j \in {\mathcal I}$.}
\begin{align} \label{initialize}
\sum_{j=1}^N \xi_j(0)=\sum_{j=1}^N \psi_j(0)=\sum_{j=1}^N \phi_j(0)=0.
\end{align}
The estimator part \eqref{avali}-\eqref{chahar}, generates the internal signal for each agent and the controller part (\ref{cont}) guarantees consensus. Here the separation principle can be applied if the estimator part converges in finite time.

\begin{assumption} \label{as2}
The estimators' coefficients $\alpha, \beta,$ and $\gamma$ satisfy the following inequalities: $\alpha > \sup_{t} \norm{\frac{\partial}{\partial t}\nabla f_i(x_i,t)}_\infty, \ \beta>\sup_{t} \norm{\frac{{\partial}}{{\partial t}} H_i(x_i,t)}_\infty,$ and $\gamma > \sup_{t} \norm{\frac{\partial^2}{\partial t^2}\nabla f_i(x_i,t)}_\infty, \ \forall i \in {\mathcal I}$.
\end{assumption}

Assumption \ref{as2} can be satisfied if the partial derivatives of the Hessians, the first- and second-order partial derivatives of the gradient are bounded.

\begin{theorem} \label{theoestsing}
Suppose that the graph ${\mathcal G}(t)$ is connected for all $t$. If Assumptions \ref{as-exist}, \ref{as1}, and \ref{as2} and the initial condition \eqref{initialize} hold, for the system (\ref{single}) with the algorithm \eqref{avali}-\eqref{cont}, the optimization goal \eqref{goal} is achieved.
\end{theorem}

\begin{proof}
\textbf{Estimator}: It follows from Theorem 2 in \cite{boundv} that if Assumption \ref{as2} holds, then there exists a $T>0$ such that for all $t\geq T$, $\norm{w_i-\frac{1}{N}\sum_{j=1}^N \nabla f_j(x_j,t)}_2=0,$ $\norm{\theta_i-\frac{1}{N}\sum_{j=1}^N H_j(x_j,t)}_2=0,$ and $\norm{\varsigma_i-\frac{1}{N}\sum_{j=1}^N \frac{\partial}{\partial t} \nabla f_j(x_j,t)}_2 =0$. Now it follows from \eqref{chahar} that for all $t \geq T$, $S_i= -(\sum_{j=1}^N {H_j(x_j,t))}^{-1}\big( \tau\sum_{j=1}^N  \nabla f_j(x_j,t)+ \sum_{j=1}^N\frac{\partial }{\partial t}\nabla  f_j(x_j,t)\big),$ $\forall i \in \mathcal{I}$. Note that for $t \geq T$, $\theta_i$ is nonsingular without projection due to Assumption \ref{as1} and hence the projection operation simply returns $\theta_i$ itself.
Till now we have shown that all agents generate the internal signal $S_i$, where $S_i=S_j, \forall i,j \in {\mathcal I}$, in finite time.

\textbf{Controller}: 
Note that $\forall \ t\geq T, \ S_i=S_j, \ \forall i,j \in {\mathcal I},$ denoted as $\bar{S}$. For $t\geq T,$ using \eqref{cont} for \eqref{single}, we have
\begin{equation} \label{contestsingleafterT} \small
\begin{split}
\dot{x}_i=&-\sum_{j \in N_i(t)} \text{sig}(x_i-x_j)^{\eta}+\bar{S}.
\end{split} \normalsize
\end{equation}
For $t\geq T$, rewriting \eqref{contestsingleafterT} using new variables $\tilde{x}_i= x_i- \int_T^t \bar{S} \text{dt}$, we have
\begin{equation} \label{snglcons}
\begin{split}
\dot{\tilde{x}}_i=&-\sum_{j \in N_i(t)}\text{sig}(\tilde{x}_i-\tilde{x}_j)^{\eta}.
\end{split}
\end{equation}
It is proved in \cite{WangSingle} that using \eqref{snglcons}, there exists a time $T'$ such that $\tilde{x}_i=\tilde{x}_j,\ \forall i,j \in {\mathcal I}$. As a result we have $x_i=x_j,\ \forall i,j \in {\mathcal I},$ and $\dot{x}_i = \bar{S},  \ \forall t\geq T+T'$. Now, it is easy to see that according to \eqref{sum-min} the optimization goal \eqref{goal} is achieved.
\end{proof}
\begin{remark} \label{assumption-explanation}
Satisfying the conditions mentioned in Assumption \ref{as2} might be restrictive but they hold for an important class of cost functions. For example, if the agents' cost functions are in the form of $f_i(x_i,t)=(a_ix_i+g_i(t))^{2}$, where the Hessians are not equal, the above conditions are equivalent to the conditions that $\norm{g_i(t)}_2, \norm{\dot{g}_i(t)}_2,$ and $\norm{\ddot{g}_i(t)}_2$ are bounded. This is applicable to a vast class of time-varying functions, $g_i(t)$, such as $\text{sin}(t), e^{-t}\text{cos}(t), \frac{1}{1+t} $ and $\text{tanh}(t)$. \end{remark}

\begin{remark} \label{pros-cons}
The algorithm introduced in (\ref{avali})-(\ref{cont}) just requires that Assumptions \ref{as1} and \ref{as2} hold. Note that in Assumption \ref{as1}, it is not required that each agent's cost function $f_i(x,t)$ has invertible Hessian but instead their sum, which is weaker than Assumption \ref{as1each}. In contrast, for the algorithm \eqref{usingledis1}, not only Assumption \ref{as1each} and the conditions mentioned in Remark \ref{Remarkboundsingle} have to be satisfied for each individual function $f_i(x_i,t)$, it requires the agents' Hessians to be equal.
However, in the algorithm \eqref{usingledis1} the agents just need their own positions and the relative positions between themselves and their neighbors.
In some applications, these pieces of information can be obtained by sensing; hence the communication necessity might be eliminated.
In contrast, in the algorithm (\ref{avali})-(\ref{cont}) each agent must communicate three variables $w_i, \varsigma_i$ and $\psi_i$ with its neighbors, which necessitates the communication requirement. 
\end{remark}
\section{Time-Varying Convex Optimization For Double-Integrator Dynamics}  \label{sec:double}
In this section, we study the convex optimization problem with time-varying cost functions for double-integrator dynamics.
In some applications, it might be more realistic to model the equations of motion of the agents with double-integrator dynamics, i.e., mass-force model, to take into account the effect of inertia. Unlike single-integrator dynamics, in the case of double-integrator dynamics, the agents' positions and velocities at each time must be determined properly such that the team cost function is minimized. However, there is only direct control on each agent's acceleration and hence there exist new challenges. As a first step, in Subsection \ref{subsecdoublecent}, a centralized algorithm will be introduced.
\subsection{Centralized Time-Varying Convex Optimization} \label{subsecdoublecent}
In this subsection, we focus on the time-varying convex optimization problem of \eqref{fcent} for double-integrator dynamics
\begin{equation} \small \label{dbl}
\dot{x}(t) = v(t)\ \quad 
\dot{v}(t) = u(t),
\end{equation} \normalsize
where $x, v, u \in  \mathbb{R}^{m}$ are the position, velocity, and control input, respectively. Our goal is to design the control input $u$ to minimize the cost function $f_0(x,t)$.
In Theorem \ref{theoremdoublecentral}, an algorithm will be proposed to solve the problem defined by (\ref{fcent}) and (\ref{dbl}). The control input is proposed for \eqref{dbl} as
\begin{equation} \label{udoublecent}
\begin{split}
u=-H^{-1}_0 (x,t)\big( \frac{\partial}{\partial t}\frac{d }{dt}\nabla f_0(x,t)+&\frac{d }{dt}\nabla f_0(x,t) \big)\\
- H_0(x,t) \nabla f_0(x,t)+ \bigg(H_0^{-1}(x,t) &[\frac{d}{dt} H_0(x,t)] H^{-1}_0(x,t)\bigg) \\
\big( \frac{\partial }{\partial t}\nabla f_0(x,t)&+\nabla f_0(x,t) \big).
\end{split}
\end{equation}

\begin{theorem} \label{theoremdoublecentral}
Suppose that $f_0(x,t)$ satisfies Assumptions \ref{as-exist} and \ref{as1}. Using  \eqref{udoublecent} for \eqref{dbl}, $x(t)$ converges to the optimal trajectory $x_0^*(t)$, the minimizer of \eqref{fcent}, i.e. $\lim_{t \to \infty} [x(t)-x_0^*(t)] = 0$.
\end{theorem}\normalsize
\begin{proof}
Define the Lyapunov function candidate $W=\frac{1}{2} \nabla f_0^T(x,t) \nabla f_0(x,t)+\frac{1}{2} (S_0-v)^T (S_0-v),$ where
$S_0={H_0}^{-1}(x,t) \big( \frac{\partial }{\partial t} \nabla f_0(x,t) + \nabla f_0(x,t)\big)$. The derivative of $W$ along the system \eqref{dbl} with the control input \eqref{udoublecent} is obtained as
\begin{equation} \label{dwdblcent}
\begin{split}
\dot{W} &=  \nabla f_0^T(x,t) (H_0(x,t) v+\frac{\partial }{\partial t} \nabla f_0(x,t))\\&+(  S_0 - v)^T (\dot{S}_0 -u )= -\nabla f_0^T(x,t) \nabla f_0(x,t)
\end{split}
\end{equation}\normalsize
Therefore, $\dot{W}< 0$ for $\nabla f_0(x,t) \neq 0$. Now, having $W\geq 0$ and $\dot{W} \leq 0$, we can conclude that $\nabla f_0(x,t), S_0-v \in {\mathcal L}_{\infty}$. By integrating both sides of \eqref{dwdblcent}, we can see that  $\nabla f_0(x,t) \in \mathcal{L}_{2}$. Now, applying Barbalat's Lemma \cite{sheida22j}, we obtain that $\nabla f_0(x,t)$ will converge to zero asymptomatically. Then by using Lemma \ref{lem-intro-gradzero} and under Assumption \ref{as-exist}, it is easy to see that $f_0$ will be minimized, where $x(t)$ converges to the optimal trajectory $x_0^*(t)$.        
\end{proof}

The results from Lemma \ref{theoremdoublecentral} can be extended to minimize the convex function $\sum_{i=1}^N f_i(x,t)$. If Assumption \ref{as1} holds, with
\begin{equation}\footnotesize \label{sum-min-double1} 
\begin{split}
u= - \big(\sum_{i=1}^N H_i(x,t)&\big)^{-1}\big(\sum_{i=1}^N \frac{\partial}{\partial t} \frac{d }{dt}\nabla f_i(x,t)+\sum_{i=1}^N \frac{d }{dt}\nabla f_i(x,t) \big)\\
-(\sum_{i=1}^N H_i(x,t))& (\sum_{i=1}^N\nabla f_i(x,t))\\
+\bigg([\sum_{i=1}^N H_i(x,t)&]^{-1} [\sum_{i=1}^N\frac{d}{dt} H_i(x,t)][\sum_{i=1}^N H_i(x,t)]\bigg)\\
&\big( \sum_{i=1}^N\frac{\partial }{\partial t}\nabla f_i(x,t)+ \sum_{i=1}^N \nabla f_i(x,t)\big)\\
\end{split}
\end{equation} \normalsize
for \eqref{dbl}, the function $\sum_{i=1}^N f_i(x,t)$ is minimized. Unfortunately, \eqref{sum-min-double1} is a centralized solution relying on the knowledge of all $f_i(x,t), i\in \mathcal{I}$. 
In Subsections \ref{subsecdbldis1} and \ref{subsecdblestim}, \eqref{sum-min-double1} will be exploited to propose two algorithms for solving the time-varying convex optimization problem for double-integrator dynamics in a distributed manner.

\subsection{Distributed Time-Varying Convex Optimization Using Neighbors' Positions and Velocities} \label{subsecdbldis1}
In what follows, we focus on solving the distributed time-varying convex optimization problem (\ref{costdis}) for agents with double-integrator dynamics 
\begin{equation} \label{double}  \small
\dot{x}_i(t) = v_i(t)\ \quad 
\dot{v}_i(t) = u_i(t), \quad i \in \mathcal{I}, \end{equation}
where $x_i, v_i \in  \mathbb{R}^{m}$ are, respectively, the position and velocity, and $u_i \in  \mathbb{R}^{m}$ is the control input of agent $i$. In this subsection, an algorithm with adaptive gains will be proposed, where each agent has access to only its own position and the relative positions and velocities between itself and its neighbors. The control input is proposed for \eqref{double} as
\begin{equation} \label{udbldis1}
\begin{split}
u_i=&-\sum_{j \in N_i} \mu (x_i-x_j)+ \alpha (v_i-v_j)\\
& - \sum_{j \in N_i} \beta_{ij} \text{sgn}(\gamma  (x_i-x_j)+ \zeta (v_i-v_j)) +\phi_i \\
\dot{\beta}_{ij}=& \norm{\gamma(x_i-x_j)+\zeta(v_i-v_j)}_1, \ \ j \in N_i,
\end{split}
\end{equation}
where 
\begin{equation} \label{phi-double}
\begin{split} \small
\phi_i\triangleq -H_i^{-1} (x_i,t)\big( \frac{\partial}{\partial t}&\frac{d }{dt}\nabla f_i(x_i,t)+\frac{d }{dt}\nabla f_i(x_i,t) \big)\\
- H_i(x_i,t) \nabla f_i(x_i&,t)\\
+ \bigg(H_i^{-1}(x_i,t) [\frac{d}{dt} &H_i(x_i,t)] H^{-1}_i(x_i,t)\bigg) \\
&\big( \frac{\partial }{\partial t}\nabla f_i(x_i,t)+\nabla f_i(x_i,t) \big),
\end{split}
\end{equation}\normalsize
where $\mu, \alpha, \gamma$ and $\zeta$ are positive coefficients, and $\beta_{ij}$ is a varying gain with $\beta_{ij}(0)=\beta_{ji}(0)\geq 0$. Note that $\phi_i$ depends on only agent $i$'s position and velocity. Furthermore all terms in \eqref{phi-double} are assumed to exist.
Define agent $i$'s position and velocity consensus error as, respectively, $e_{X_i}=x_i-\frac{1}{N}\sum_{\ell=1}^N x_\ell$ and $e_{V_i}=v_i-\frac{1}{N}\sum_{\ell=1}^N v_\ell$.
Let $X=[x_1^T,x_2^T,...,x_N^T]^T$ and $V=[v_1^T,v_2^T,...,v_N^T]^T$. Define the consensus error vectors for position and velocity as
\begin{equation} \label{exev}
e_X(t)=(\Pi \otimes I_m)X, \ e_V(t)=(\Pi \otimes I_m)V.
\end{equation}
As discussed in Subsection \ref{subsecsingledis1}, it is easy to see that $e_X(t)=0, e_V(t)=0$ if and only if $x_i=x_j , v_i=v_j, \ \forall i,j \in {\mathcal I}$. Also let $\Phi=[\phi_1^T,\phi_2^T,...,\phi_N^T]^T$.

\begin{assumption} \label{as-bound-double}
With $\phi_i$ defined in  \eqref{phi-double}, there exists a positive constant $\bar{\phi}$ such that $\norm{\phi_i -\phi_j}_2 \leq \bar{\phi}, \  \forall i,j \in {\mathcal I},$ and $\forall t$.
\end{assumption}
In Proposition \ref{propdbldis1}, we will show that the agents reach consensus using (\ref{udbldis1}). Then this result will be used in Theorem \ref{theorem1} to show that the agents minimize the team cost function as $t\rightarrow\infty$.  
\begin{prop} \label{propdbldis1}
Suppose that the graph ${\mathcal G}$ is connected, and Assumption \ref{as-bound-double} and $\frac{\gamma}{\alpha \zeta}< \lambda_2[L]$ hold. The system \eqref{double} with the algorithm (\ref{udbldis1}) reaches consensus, i.e, $x_i=x_j, v_i=v_j,\ \forall i,j \in {\mathcal I},$ as $t\rightarrow\infty$.   
\end{prop}

\begin{proof}
The closed-loop system \eqref{double} with the control input \eqref{udbldis1} can be recast into a compact form as
\begin{equation} \label{closedbl1}
\begin{cases} \dot{X} =& V \\ 
\dot{V} =&-(L\otimes I_m)(\mu X+\alpha V)+\Phi\\ &- (D' \otimes I_m) \text{sgn}\big( [D^T \otimes I_m][\gamma X+\zeta V]\big), \end{cases}
\end{equation}
where $D'$ is defined in Definition \ref{def-new-lap}, and $D$  and $L$ are defined in Section II. Now, we can rewrite (\ref{closedbl1}) as
\begin{equation} \label{errordbldis1}
\begin{cases} \dot{e}_X =& e_V \\ 
\dot{e}_V =&-(L\otimes I_m)(\mu e_X+\alpha e_V)+(\Pi\otimes I_m) \Phi\\ &- (D' \otimes I_m) \text{sgn}\big( [D^T \otimes I_m][\gamma X+\zeta V]\big).
\end{cases} 
\end{equation}

Define the function \begin{equation} \label{lyapdbldis}
W= \frac{1}{2}\left( {\begin{array}{cc} e_X \\ e_V \\ \end{array} } \right)^T P \left( {\begin{array}{cc} e_X \\ e_V \\ \end{array} } \right)+\frac{1}{2}\sum_{i=1}^{N}\sum_{j \in N_i} (\beta_{ij}-\bar{\beta})^2,
\end{equation} where $P=\left( {\begin{array}{cc} (\alpha \gamma+ \mu \zeta )(L\otimes I_m) & \gamma I_{mN} \\ \gamma I_{mN} & \zeta I_{mN} \\ \end{array} } \right)$ and $\bar{\beta}>0$ is to be selected. To prove the positive definiteness of $P$, we define $\hat{P}=\left( {\begin{array}{cc} (\alpha \gamma+ \mu \zeta )(\lambda_2[L] I_{mN}) & \gamma I_{mN} \\ \gamma I_{mN} & \zeta I_{mN} \\ \end{array} } \right)$. By using Lemma \ref{eigvalue}, we obtain that $\hat{P} \leq P$. Hence we just need to show $\hat{P}>0$. Now, applying Lemma \ref{schur}, $\hat{P}>0$ if $\zeta (\alpha \gamma +\mu \zeta)\lambda_2[L] I_N - \gamma^2 I_N>0$, which is equivalent to $\frac{\gamma}{\alpha \zeta}< \lambda_2[L]$. 

The time derivative of $W$ along (\ref{errordbldis1}) can be obtained as
\begin{equation*}
\begin{split}
&\dot{W}=-\gamma \mu e_X^T (L\otimes I_m) e_X + e_V^T (\gamma I_{mN} - \alpha \zeta L\otimes I_m) e_V\\
&-\big(\gamma e_X+\zeta e_V\big)^T \bigg([D' \otimes I_m]\text{sgn} \big([D^T\otimes I_m] [\gamma e_X+\zeta e_V]\big)\bigg)\\
& +\big(\gamma e_X+\zeta e_V\big)^T( \Pi\otimes I_m) \Phi\nonumber
+\frac{1}{2}\sum_{i=1}^{N}\sum_{j \in N_i} (\beta_{ij}-\bar{\beta})\dot{\beta}_{ij}\\
&=-\gamma \mu e_X^T (L\otimes I_m) e_X + e_V^T (\gamma I_{mN} - \alpha \zeta L\otimes I_m) e_V\\
& -\frac{1}{2} \sum_{i=1}^{N}\sum_{j \in N_i} \beta_{ij} \norm{\gamma(e_{X_i}-e_{X_j})+\zeta(e_{V_i}-e_{V_j})}_1 \\
&+\frac{1}{N}\sum_{i=1}^{N}\sum_{j=1}^{N}\big[\gamma(e_{X_i}-e_{X_j})+\zeta(e_{V_i}-e_{V_j})\big]\phi_i \\
&+\frac{1}{2}\sum_{i=1}^{N}\sum_{j \in N_i} (\beta_{ij}-\bar{\beta})\dot{\beta}_{ij}\\ 
&\leq -\gamma \mu e_X^T (L\otimes I_m) e_X + e_V^T (\gamma I_{mN} - \alpha \zeta L\otimes I_m) e_V\\
&- \frac{1}{2} \sum_{i=1}^{N}\sum_{j \in N_i}\beta_{ij} \norm{\gamma(e_{X_i}-e_{X_j})+\zeta(e_{V_i}-e_{V_j})}_1\\
&+\frac{1}{2N}\sum_{i=1}^{N}\sum_{j=1}^{N}\norm{\gamma(e_{X_i}-e_{X_j})+\zeta(e_{V_i}-e_{V_j})}_1\norm{\phi_i-\phi_j}\\
& +\frac{1}{2}\sum_{i=1}^{N}\sum_{j \in N_i}(\beta_{ij}-\bar{\beta})\norm{\gamma(e_{X_i}-e_{X_j})+\zeta(e_{V_i}-e_{V_j})}_1
\end{split}
\end{equation*}
\begin{equation*}
\begin{split}
&\leq -\gamma \mu e_X^T (L\otimes I_m) e_X + e_V^T (\gamma I_{mN} - \alpha \zeta L\otimes I_m) e_V\\
&+( \frac{(N-1)\bar{\phi}}{4}-\frac{\bar{\beta}}{2}) \sum_{i=1}^{N}\sum_{j \in N_i}\norm{\gamma(e_{X_i}-e_{X_j})+\zeta(e_{V_i}-e_{V_j})}_1,
\end{split}
\end{equation*} \normalsize
where the last inequality is obtained because $\mathcal{G}$ is connected and Assumption \ref{as-bound-double} holds. The term $ e_V^T (\gamma I_{mN} - \alpha \zeta L\otimes I_m) e_V<0$ if $\gamma I_{N} - \alpha \zeta L<0$. By applying Lemma \ref{eigvalue}, we know that $\gamma I_{N} - \alpha \zeta L<0$ if $\gamma - \alpha \zeta \lambda_2[L]<0$. Select $\bar{\beta}$ such that $\bar{\beta}  \geq\frac{ (N-1)\bar{\phi}}{2}$. Using an argument similar to \eqref{local-lambda2}, we have
\begin{equation} \label{local-lambdadbl}
\begin{split} \small
\dot{W}\leq& ( \frac{(N-1)\bar{\phi}}{4}-\frac{\bar{\beta}}{2}) \sum_{i=1}^{N}\sum_{j \in N_i}\norm{\gamma(e_{X_i}-e_{X_j})+\zeta(e_{V_i}-e_{V_j})}_1\\ 
=&(\frac{(N-1)\bar{\phi}}{2}-\bar{\beta})\big(\gamma e_X+\zeta e_V\big)^T\\ 
&\bigg([D' \otimes I_m]\text{sgn} ([D^T\otimes I_m] [\gamma e_X+\zeta e_V])\bigg) \\
&\leq(\frac{(N-1)\bar{\phi}}{2}-\bar{\beta}) \norm{(D^T \otimes I_m)(\gamma e_X+\zeta e_V)}_1\\
&\leq (\frac{(N-1)\bar{\phi}}{2}-\bar{\beta})\sqrt{\lambda_{2}[L]} \norm{(\gamma e_X+\zeta e_V)}_2< 0,
\end{split}
\end{equation}\normalsize
where in the last inequality the fact that $L=DD^T$ and Lemma \ref{eigvalue} have been used. Therefore, having $W\geq 0$ and $\dot{W} \leq 0$, we can conclude that $e_X, e_V \in {\mathcal L}_{\infty}$. By integrating both sides of \eqref{local-lambdadbl}, we can see that  $e_X, e_V \in \mathcal{L}_{2}$. Now, applying Barbalat's Lemma \cite{sheida22j}, we obtain that $e_X$ and $e_V$ will converge to zero asymptotically and hence the agents reach consensus as $t\rightarrow\infty$.
\end{proof}

\begin{theorem} \label{theorem1}
Suppose that the graph ${\mathcal G}$ is connected, and Assumptions \ref{as-exist}, \ref{as1each} and \ref{as-bound-double} hold. If $H_i(x_i,t)=H_j(x_j,t), \ \forall i,j \in {\mathcal I}$, and $\frac{\gamma}{\alpha \zeta}< \lambda_2[L]$
hold, by employing the algorithm (\ref{udbldis1}) for the system (\ref{double}), the optimization goal \eqref{goal} is achieved.
\end{theorem}

\begin{proof}
Define the Lyapunov function candidate
\begin{equation}
\begin{split}
W&=\frac{1}{2}(\sum_{j=1}^N \nabla f_j(x_j,t))^T (\sum_{j=1}^N \nabla f_j(x_j,t))\\&+\frac{1}{2} (\sum_{j=1}^N v_j -\sum_{j=1}^N S_j )^T (\sum_{j=1}^N v_j -\sum_{j=1}^N S_j ),
\end{split}
\end{equation}
where $S_j \defeq {H_j}^{-1}(x_j,t) \big( \frac{\partial }{\partial t} \nabla f_j(x_j,t) +  \nabla f_j(x_j,t)\big)$.
Note that ${\mathcal G}$ is undirected. By summing both sides of the closed-loop system \eqref{double} with the controller \eqref{udbldis1}, we have $\sum_{j=1}^N \dot{v}_j = \sum_{j=1}^N \phi_j$. The time derivative of $W$ along with the system defined by \eqref{double} and \eqref{udbldis1} can be obtained as

\begin{equation*} \small
\begin{split}
\dot{W}=& (\sum_{j=1}^N \nabla f_j(x_j,t))^T \big(\sum_{j=1}^N H_j(x_j,t) v_j+ \sum_{j=1}^N \frac{\partial }{\partial t} \nabla f_j(x_j,t)\big)\\&+ (\sum_{j=1}^N v_j -\sum_{j=1}^N S_j )^T (\sum_{j=1}^N \dot{v}_j -\sum_{j=1}^N \dot{S}_j )\\
=& (\sum_{j=1}^N \nabla f_j(x_j,t))^T \big(\sum_{j=1}^N H_j(x_j,t) v_j+ \sum_{j=1}^N \frac{\partial }{\partial t} \nabla f_j(x_j,t)\big)\\&+ (\sum_{j=1}^N v_j -\sum_{j=1}^N S_j )^T (\sum_{j=1}^N \phi_j -\sum_{j=1}^N \dot{S}_j ) \\
=& (\sum_{j=1}^N \nabla f_j(x_j,t))^T \big(\sum_{j=1}^N H_j(x_j,t) v_j+ \sum_{j=1}^N \frac{\partial }{\partial t} \nabla f_j(x_j,t)\big)\\&- (\sum_{j=1}^N v_j -\sum_{j=1}^N S_j )^T (\sum_{j=1}^N H_j(x_j,t)\nabla f_j(x_j,t)). \nonumber
\end{split}
\end{equation*} \normalsize
Now, under the assumption of identical Hessians, we have
\begin{equation} \small \label{wdot}
\begin{split}
\dot{W}=&(\sum_{j=1}^N \nabla f_j(x_j,t))^T \big(H_j(x_j,t)[\sum_{j=1}^N  v_j]+ \sum_{j=1}^N \frac{\partial }{\partial t} \nabla f_j(x_j,t)\big)\\
& - (\sum_{j=1}^N v_j -\sum_{j=1}^N S_j )^T \big(H_j(x_j,t)[\sum_{j=1}^N \nabla f_j(x_j,t)]\big)\\
=&-(\sum_{j=1}^N \nabla f_j(x_j,t))^T (\sum_{j=1}^N \nabla f_j(x_j,t)).
\end{split}
\end{equation} \normalsize
Therefore, $\dot{W}<0$ for $\sum_{j=1}^N \nabla f_j(x_j,t) \neq 0$. Now, having $W\geq 0$ and $\dot{W} \leq 0$, we can conclude that $\sum_{j=1}^N \nabla f_j(x_j,t), \big(\sum_{j=1}^N v_j -\sum_{j=1}^N S_j\big) \in \mathcal{L}_{\infty}$. By integrating both sides of \eqref{wdot}, we can see that  $\sum_{j=1}^N \nabla f_j(x_j,t) \in \mathcal{L}_{2}$. Now, applying Barbalat's Lemma \cite{sheida22j}, we obtain that $\sum_{i=1}^N \nabla f_i(x,t)$ will converge to zero asymptomatically. We also have $x_i=x_j,  v_i=v_j, \ \forall i,j \in {\mathcal I}$ as $t\rightarrow \infty$ from Proposition \ref{propdbldis1}.
Now, under the assumption that $\sum_{i=1}^N f_i(x,t)$ is convex and using Lemma \ref{lem-intro-gradzero}, it is easy to see from Assumption \ref{as-exist} that as $t\rightarrow \infty$, $\sum_{j=1}^N f_i(x_j,t)$ will be minimized, where $x_i=x_j,\ \forall i,j \in {\mathcal I}$.
\end{proof}

\begin{remark} \label{Remarkbounddouble}
In Appendix \ref{AppendixA},  we show that Assumption \ref{as-bound-double} holds if the cost functions with identical Hessians satisfy certain conditions such that the boundedness of $\norm{x_i-x_j}_2$ and $\norm{v_i-v_j}_2$ for all $t$ guarantees the boundedness of  $\norm{\nabla f_j(x_j,t) - \nabla f_i(x_i,t) }_2$, $\norm{\frac{d }{dt}\nabla  f_j(x_j,t)-\frac{d }{dt}\nabla  f_i(x_i,t)}_2$ and $\norm{\frac{\partial^2  }{\partial t^2}\nabla  f_j(x_j,t)-\frac{\partial^2 }{\partial t^2}\nabla  f_i(x_i,t)}_2$ for all $t$.
For example, for the cost functions $f_i(x_i,t)=(ax_i+g_i(t))^{2}$ introduced in Remark \ref{Remarkboundsingle},  the boundedness of $\norm{x_i-x_j}_2$ and $\norm{v_i-v_j}_2$ for all $t$ guarantees the boundedness of  $\norm{\nabla f_j(x_j,t) - \nabla f_i(x_i,t) }_2$, $\norm{\frac{d }{dt}\nabla  f_j(x_j,t)-\frac{d }{dt}\nabla  f_i(x_i,t)}_2$ and $\norm{\frac{\partial^2  }{\partial t^2}\nabla  f_j(x_j,t)-\frac{\partial^2 }{\partial t^2}\nabla  f_i(x_i,t)}_2$ for all $t$, if $\norm{g_i(t) -g_j(t)}_2, \ \norm{\dot{g}_i(t) -\dot{g}_j(t)}_2$ and $\norm{\ddot{g}_i(t) -\ddot{g}_j(t)}_2$ are bounded.
Hence Assumption \ref{as-bound-double} holds for $f_i(x_i,t)=(ax_i+g_i(t))^{2}$,  if $\norm{g_i(t) -g_j(t)}_2, \ \norm{\dot{g}_i(t) -\dot{g}_j(t)}_2$ and $\norm{\ddot{g}_i(t) -\ddot{g}_j(t)}_2, \forall t$ and $\forall i,j \in {\mathcal I}$ are bounded.
\end{remark}
\begin{remark} \label{remark-notidentHessian}
The result in Theorem 4.2 can be extended to a class of cost functions whose Hessians have the same structure rather than being identical under a certain additional assumption. Particularly, the assumption that $H_i(x_i, t) = H_j(x_j , t), \forall t$ and $\forall i, j \in \mathcal{I},$ can be replaced with $H_i(z,t)=H_j(z,t), \ \forall z, t$ and $\forall i, j \in \mathcal{I}$, with an additional assumption that $\norm{\nabla f_i(x_i,t)}_2$ and $\norm{ \frac{\partial }{\partial t} \nabla f_i(x_i,t)}_2, \  \forall i \in \mathcal{I}$ and $\forall t,$ are bounded.
\end{remark}

To relax the assumption on Hessians, an estimator-based algorithm will be introduced in Subsection \ref{subsecdblestim}, where the agents can have cost functions with nonidentical Hessians.
 
\subsection{Estimator-Based Distributed Time-Varying Convex Optimization} \label{subsecdblestim}
\normalsize
In this subsection, an estimator-based algorithm is designed to solve the problem \eqref{costdis} for double-integrator dynamics \eqref{double}.
In this algorithm, each agent calculates \eqref{sum-min-double1} in a distributed manner. Similar to Subsection \ref{secsingledis2}, distributed average tracking is used as a tool to estimate the unknown variables in \eqref{sum-min-double1}.
Each agent generates an estimate of \eqref{sum-min-double1}. Then by using the control input $u_i(t)$, each agent tracks its estimated signal while reaching consensus. The proposed algorithm for the system (\ref{double}) is given by
\begin{equation} \label{koldbl1}
\dot{\xi}_i= \kappa \sum_{j \in N_i(t)} \text{sgn}(w_j - w_i),\ \ w_i= \xi_i+  \psi_i 
  \end{equation}
   \begin{equation} \label{koldbl2}
\dot{\phi}_i=\rho \sum_{j \in N_i(t)} \text{sgn}(\varsigma_j -\varsigma_i),\ \ \ \varsigma_i=\phi_i+\theta_i 
  \end{equation}
\begin{equation}
  \begin{split} \label{chahardbl} \small
&S_i= \varsigma_{i1}^{-1} \varsigma_{i2} \varsigma_{i1} \big( w_{i1}+w_{i2} \big)- \varsigma_{i1}^{-1} \big( w_{i3}+w_{i4} \big) - \varsigma_{i1} w_{i1} \\
  \end{split}  \normalsize
  \end{equation}
  \begin{equation} \label{contestdbl}
\begin{split}
u_i&=-\sum_{j \in N_i(t)}\text{sig}(x_i-x_j)^{\alpha_1} -\sum_{j \in N_i(t)}\text{sig}(v_i-v_j)^{\alpha_1} +S_i, 
\end{split}
\end{equation}
where \begin{equation} \label{alpha12}
0<\alpha_1<1,\ \ \ \ \ \alpha_2=\frac{2\alpha_1}{\alpha_1+1}
\end{equation} 
and
\begin{equation*}
\psi_i=\begin{pmatrix}
  {\nabla} f_i(x_i,t) \\ \vspace{+.15cm}
    \frac{{\partial}}{{\partial t}}\nabla f_i(x_i,t) \\ \vspace{+.15cm}
    \frac{{d}}{{d t}}\nabla f_i(x_i,t) \\ \vspace{+.15cm}
    \frac{\partial}{\partial t}\frac{d}{dt}\nabla f_i(x_i,t)
        \end{pmatrix}, \
\theta_i=\begin{pmatrix}
   H_i(x_i,t)\\ \vspace{+.15cm}
    \frac{{d}}{{d t}}H_i(x_i,t)
        \end{pmatrix},
\end{equation*}
and $\kappa$ and $\rho$ are positive constant coefficients to be selected. Eqs. \eqref{koldbl1} and \eqref{koldbl2} are distributed average tracking estimators, where the estimated variables $w_i$ and $\varsigma_i$ can be redefined as $w_i^T=(w_{i1}^T, \hdots, w_{i4}^T)$ and $\varsigma_i^T=(\varsigma_{i1}^T, \varsigma_{i2}^T)$ with $w_{ij} \in  \mathbb{R}^m ,\varsigma_{ik} \in  \mathbb{R}^{m\times m}, \forall i \in  {\mathcal I}, j=1,...,4, k=1,2$. In implementing \eqref{chahardbl}, $\varsigma_{i1}$ can be projected on the space of positive-definite matrices, which ensures that $\varsigma_{i1}$ remains nonsingular. The initial values of the internal states $\xi_i$ and $\phi_i$ are chosen such that the condition \footnote{As a special case the initial values can be chosen as $\xi_j(0)=\phi_j(0)=0, \forall j \in {\mathcal I}$.}
\begin{align} \label{initializedbl} \small
\sum_{j=1}^N \xi_j(0)=\sum_{j=1}^N \phi_j(0)=0
\end{align} \normalsize
hols.
\begin{assumption} \label{ass2}
The coefficients $\kappa,$ and $\rho$ satisfy the following inequalities: $\kappa > \sup_{t} \norm{\psi_i}_\infty$ and $\rho > \sup_{t} \norm{\theta_i}_\infty, \ \forall i \in {\mathcal I}$.
\end{assumption}

These assumptions can be satisfied if the graph ${\mathcal G}(t)$ is connected for all time, the gradients, the derivatives of the Hessians and gradients, and the partial derivatives of the gradients' derivatives are bounded. Although these assumptions seem restrictive, they can be satisfied for many cost functions. For example, for $f_i(x_i,t)=(a_ix_i+g_i(t))^{2}$ introduced in Remark \ref{assumption-explanation}, the above conditions are equivalent to the conditions that $\norm{g_i(t)}_2, \norm{\dot{g}_i(t)}_2,$ and $\norm{\ddot{g}_i(t)}_2$ are bounded.

\begin{theorem}
Suppose that the graph ${\mathcal G}(t)$ is connected for all $t$ and Assumptions \ref{as-exist} and \ref{as1} hold. If Assumption \ref{ass2} and the initial condition \eqref{initializedbl} hold, by employing the algorithm  \eqref{koldbl1}-\eqref{contestdbl} for the system (\ref{double}), the optimization goal \eqref{goal} is achieved.
\end{theorem}

\begin{proof}
\textbf{Estimator}: It follows from Theorem 2 in \cite{boundv} that if $\kappa > \sup_{t} \norm{\psi_i}_\infty, \rho > \sup_{t} \norm{\theta_i}_\infty, \ \forall i \in {\mathcal I}$ and the graph ${\mathcal G}(t)$ is connected for all $t$, then employing \eqref{koldbl1} and \eqref{koldbl2}, there exists a $T>0$ such that for all $t\geq T$, $\norm{w_i-\frac{1}{N}\sum_{i=1}^N \psi_i}_2=0,  \norm{\varsigma_i-\frac{1}{N}\sum_{i=1}^N \theta_i}_2=0.$ Note that for $t \geq T$, $\varsigma_{i1}$ is nonsingular without projection due to Assumption \ref{as1} and hence the projection operation simply returns $\varsigma_{i1}$ itself. Now from \eqref{chahardbl}, for $t \geq T$, the estimated signal $S_i$ satisfies 
\begin{equation}\footnotesize \label{Sidouble} 
\begin{split}
S_i= -\big(\sum_{i=1}^N H_i(x_i,t)\big) \big(\sum_{i=1}^N&\nabla f_i(x_i,t)\big)\\
-\big(\sum_{i=1}^N H_i(x_i,t)\big)^{-1}\big(\sum_{i=1}^N &\frac{\partial}{\partial t} \frac{d }{dt}\nabla f_i(x_i,t)+\sum_{i=1}^N \frac{d }{dt}\nabla f_i(x_i,t) \big) \\
+\bigg([\sum_{i=1}^N H_i(x_i,t)]^{-1} [\sum_{i=1}^N&\frac{d}{dt} H_i(x_i,t)][\sum_{i=1}^N H_i(x_i,t)]\bigg)\\
& \big(\sum_{i=1}^N \nabla f_i(x_i,t)+\sum_{i=1}^N\frac{\partial }{\partial t}\nabla f_i(x_i,t)\big),\\
\end{split}
\end{equation} \normalsize
which shows that each agent has an estimate of \eqref{sum-min-double1}, where $S_i=S_j,\ \forall i,j \in {\mathcal I} , \forall t\geq T$.

\textbf{Controller}: Note that for $t\geq T, \ S_i=S_j, \ \forall i,j \in {\mathcal I},$ denoted as $\bar{S}$. For $t\geq T$ using \eqref{contestdbl} for \eqref{double}, we have
\begin{equation} \label{contestdblafterT} \small
\begin{split}
\dot{v}_i=&-\sum_{j \in N_i(t)}\text{sig}(x_i-x_j)^{\alpha_1} -\sum_{j \in N_i(t)}\text{sig}(v_i-v_j)^{\alpha_2}+\bar{S}.
\end{split} \normalsize
\end{equation}
For $t\geq T$, rewriting \eqref{contestdblafterT} using new variables $\tilde{x}_i= x_i- \int_T^t\int_T^t \bar{S} \text{dt}$ and $\tilde{v}_i= v_i- \int_T^t \bar{S} \text{dt}$, we have
\begin{equation} \label{dblcons}
\begin{split}
\dot{\tilde{v}}_i=&-\sum_{j \in N_i(t)}\text{sig}(\tilde{x}_i-\tilde{x}_j)^{\alpha_1} -\sum_{j \in N_i(t)}\text{sig}(\tilde{v}_i-\tilde{v}_j)^{\alpha_2}.
\end{split}
\end{equation}
It is proved in \cite{HongJ} that using \eqref{dblcons}, there exists a time $T'$ such that $\tilde{x}_i=\tilde{x}_j,\ \tilde{v}_i=\tilde{v}_j,\ \forall i,j \in {\mathcal I}$. As a result we have $x_i=x_j,\ v_i=v_j, \ \forall i,j \in {\mathcal I},$ and $\dot{v}_i = \bar{S}$, $\forall \ t\geq T+T'$. Now, it is easy to see that according to \eqref{sum-min-double1} and \eqref{Sidouble} and Assumption \ref{as-exist}, the optimization goal \eqref{goal} is achieved.
\end{proof}
\begin{remark} \label{pros-cons-dbl}
In the algorithm \eqref{koldbl1}-\eqref{contestdbl}, it is just required that Assumptions \ref{as1} and \ref{ass2} hold, where Assumption \ref{as1each} does not necessarily hold. However, each agent must communicate two variables $\psi_i \in  \mathbb{R}^{4m}$ and $\theta_i \in  \mathbb{R}^{2m \times m}$ with its neighbors, which necessitates the communication requirement. On the other hand, for the algorithm \eqref{udbldis1}, not only Assumptions \ref{as1each} and the conditions in Remark \ref{Remarkbounddouble} have to be satisfied for each individual function $f_i(x_i,t)$, it requires the agents' Hessians to be equal. In spite of these restrictive assumptions, using \eqref{udbldis1}, we can eliminate the necessity of communication between neighbors when the relative positions and velocities between each agent and its neighbors can be obtained by sensing.
\end{remark}

In what follows we will study how to overcome the possible chattering effect of implementing the signum function in the algorithm \eqref{udbldis1}. In Subsections \ref{subsecdblsgntimevar} and \ref{subsecdblsgninv}, two continuous control algorithms will be proposed to extend \eqref{udbldis1}.
\subsection{Distributed Time-Varying Convex Optimization Using Time-Varying Approximation of Signum Function} \label{subsecdblsgntimevar}
In this subsection, we focus on distributed time-varying convex optimization for double-integrator dynamics \eqref{double}, where a continuous control algorithm based on the boundary layer concept will be introduced. Using a continuous approximation of the signum function will reduce chattering in real applications and make the controller easier to implement.
In this algorithm, each agent needs to know its own position, velocity and the relative positions and velocities between itself and its neighbors. Define the nonlinear function $h(\cdot)$ as
\begin{equation} \label{hvaring} \small h(z)=\frac{z}{\norm{z}_2+\epsilon e^{-c t}},
\end{equation}\normalsize
where $c$ and $\epsilon$ are positive coefficients and $z \in  \mathbb{R}^m$. The nonlinear function $h(z)$ is a continuous approximation, using the boundary layer concept \cite{zhao20}, of the discontinuous function sgn$(\cdot)$. The size of boundary layer, $\epsilon e^{-c t}$ is time-varying and as $t\rightarrow\infty$ the continuous function $h(z)$ approaches the signum function. The idea of using this continuous approximation is borrowed from \cite{signaprox, newcite}.   

By replacing the signum function with a continuous approximation \eqref{hvaring}, the results presented in Subsection \ref{subsecdbldis1} are not valid anymore and it is not clear whether the new algorithm works. The reason is that the results (and the proofs) in Subsection \ref{subsecdbldis1} build upon the property of the ideal discontinuous signum function, which switches instantaneously at $0$, so that it can compensate for the effect of the inconsistent internal time-varying optimization signals among the agents so that the agents can reach consensus. However, this can no longer be achieved by its continuous replacement and further careful analysis is needed. The results and the proofs presented in this subsection are not just a simple replacement of the signum function with its approximation. Here in particular we show that with the signum function replaced with the time-varying continuous approximation, \eqref{hvaring}, it is possible to still achieve distributed optimization with zero error under certain different assumptions and conditions. The reason is that \eqref{hvaring} approaches the signum function as $t\to \infty$.

The continuous control input with adaptive gains is proposed for \eqref{double} as
\begin{equation} \label{udbldissgnvary}
\begin{split}
u_i=& -\sum_{j \in N_i} \mu (x_i-x_j)+ \alpha (v_i-v_j)\\& -\sum_{j \in N_i} \beta _{ij}\ h(\gamma  [x_i-x_j]+ \zeta [v_i-v_j]) +\phi_i,\\
\dot{\beta}_{ij}=&\big(\gamma[x_i-x_j]+\zeta[v_i-v_j]\big) h(\gamma  [x_i-x_j]+ \zeta [v_i-v_j]),
\end{split}
\end{equation}
where $\mu, \alpha, \gamma$ and $\zeta$ are positive coefficients, $\beta_{ij}$ is a varying gain with $\beta_{ij}(0)=\beta_{ji}(0)\geq 0$, and $\phi_i$ is defined as in \eqref{udbldis1}.

\begin{theorem} \label{theoremsgnvar}
Suppose that the graph ${\mathcal G}$ is connected, and \begin{align}
 \frac{\gamma}{\alpha \zeta}+ \frac{\psi}{\alpha} < \lambda_2[L] \label{con3} \\
 \frac{\mu}{2\alpha}>\psi \label{con1}\\
 \frac{\gamma}{2 \zeta}>\psi \label{con4}
\end{align}
hold, where $\psi>0$ is a parameter to be selected. If Assumptions \ref{as-exist}, \ref{as1each} and \ref{as-bound-double} hold and $H_i(x_i,t)=H_j(x_j,t), \forall t$ and $\forall i,j \in {\mathcal I}$, by employing the algorithm (\ref{udbldissgnvary}) for the system (\ref{double}), the optimization goal \eqref{goal} is achieved.
\end{theorem}

\begin{proof}
Define $e_X$ and $e_V$ as in \eqref{exev} and $y$ as $y^T=( y_1^T,\hdots, y_N^T)= \gamma e_X^T+\zeta e_V^T$, with $y_i \in  \mathbb{R}^m$. Rewriting the closed-loop system \eqref{double} using \eqref{udbldissgnvary} in terms of the consensus errors, we have
\begin{equation} \label{errordbldissgnvar}
\begin{cases} \dot{e}_X =& e_V \\ 
\dot{e}_V =&-(L\otimes I_m)(\mu e_X+\alpha e_V)\\ &-\begin{pmatrix}
 \sum_{j \in N_1}  \beta_{1j} h\big(y_1-y_j\big) \\
    \vdots \\
   \sum_{j \in N_N}  \beta_{Nj} h\big(y_N-y_j\big)
 \end{pmatrix} +(\Pi\otimes I_m) \Phi.
\end{cases} 
\end{equation} 

Define the function 
\begin{equation} \label{lypapWdbl} \small W= \frac{1}{2}\left( {\begin{array}{cc} e_X \\ e_V \\ \end{array} } \right)^T P \left( {\begin{array}{cc} e_X \\ e_V \\ \end{array} } \right)+\frac{1}{2}\sum_{i=1}^{N}\sum_{j \in N_i} (\beta_{ij}-\bar{\beta})^2,\end{equation}\normalsize
where $P=\left( {\begin{array}{cc} -2\psi \gamma I_{Nm}+(\alpha \gamma+ \mu \zeta )(L\otimes I_m) & \gamma I_{mN} \\ \gamma I_{mN} & \zeta I_{mN} \\ \end{array} } \right)$, and $\bar{\beta}>0$ is to be selected. To prove the positive definiteness of $P$, we also define $\hat{P}=\left( {\begin{array}{cc} -2\psi \gamma I_{mN}+(\alpha \gamma+ \mu \zeta )(\lambda_2[L]I_{mN}) & \gamma I_{mN} \\ \gamma I_{mN} & \zeta I_{mN} \\ \end{array} } \right)$. By using Lemma \ref{eigvalue}, we obtain that $\hat{P} \leq P$. Hence we just need to show $\hat{P}>0$. Using Lemma \ref{schur}, we know that $\hat{P}>0$ if
\begin{equation} \label{condpos1}
\begin{split}
 &-2\psi \gamma I_{N} +(\alpha \gamma+ \mu \zeta )\lambda_2[L])I_{N}>0\\
 &-2\zeta \psi \gamma I_{N}+\zeta(\alpha \gamma+ \mu \zeta )\lambda_2[L]-\gamma^2 I_N>0.\\
\end{split}
\end{equation}
Now, using conditions \eqref{con3} and \eqref{con1}, respectively, we have
\begin{equation*}
\begin{split}
 &-2\psi \gamma I_{N} +(\alpha \gamma+ \mu \zeta )\lambda_2[L]I_{N}\\
 &>-2\psi \gamma I_{N}+\frac{(\alpha \gamma+ \mu \zeta )\gamma}{\alpha \zeta}I_{N}=\frac{\gamma}{\alpha \zeta}(-2\psi \alpha \zeta+ \zeta \mu +\alpha \gamma^2) I_{N}\\
&> \frac{\gamma}{\alpha \zeta} (-\zeta \mu +\zeta \mu +\alpha \gamma)I_{N}=\frac{\gamma^2}{ \zeta} >0,
\end{split}
\end{equation*}
which guarantees that the first inequality in \eqref{condpos1} holds. Applying a similar procedure, we have
\begin{equation*}
\begin{split}
&-2\zeta \psi \gamma I_{N}+\zeta(\alpha \gamma+ \mu \zeta )\lambda_2[L]-\gamma^2 I_N\\
&>  -2\zeta \psi \gamma I_{N}+(\alpha \gamma+ \mu \zeta )\frac{\gamma}{\alpha}I_N-\gamma^2 I_N\\
&=\frac{\zeta \gamma}{\alpha}(-2\alpha \psi +\mu) I_N>0.
\end{split}
\end{equation*}
Hence $W$ is positive definite.

The time derivative of $W$ along (\ref{errordbldissgnvar}) can be obtained as
\vspace{-.1cm}
\begin{equation} \label{dreaWsgn} \small
\begin{split}
&\dot{W}=-\gamma \mu e_X^T (L\otimes I_m) e_X + e_V^T (\gamma I_{mN} - \alpha \zeta L\otimes I_m) e_V\\
&-2\psi \gamma e_X^T e_V- y^T \begin{pmatrix}
  \sum_{j \in N_1}\beta_{1j} h\big(y_1-y_j\big) \\
    \vdots \\
   \sum_{j \in N_N} \beta_{Nj} h\big(y_N-y_j\big)
 \end{pmatrix}\\
 &+y^T( \Pi\otimes I_m) \Phi+\frac{1}{2}\sum_{i=1}^{N}\sum_{j \in N_i} (\beta_{ij}-\bar{\beta})\dot{\beta}_{ij}.
\end{split}
\end{equation} \normalsize
We rewrite \eqref{dreaWsgn} as $\dot{W}=\bar{W}- \sum_{i=1}^{N}\sum_{j \in N_i}\beta_{ij} y_i h(y_i- y_j)+y^T( \Pi\otimes I_m) \Phi+\frac{1}{2}\sum_{i=1}^{N}\sum_{j \in N_i} (\beta_{ij}-\bar{\beta})\dot{\beta}_{ij}$, where $\bar{W}=\left( {\begin{array}{cc} e_X \\ e_V \\ \end{array} } \right)^T \bar{P} \left( {\begin{array}{cc} e_X \\ e_V \\ \end{array} } \right)$ and $\bar{P}=\left( {\begin{array}{cc} -\mu \gamma (L\otimes I_m) & - \psi \gamma I_{mN} \\ -\psi \gamma I_{mN} & \gamma I_{mN}-\alpha \zeta (L\otimes I_m) \\ \end{array} } \right)$.
Because the graph $\mathcal{G}$ is connected, we have
\begin{equation*} \small
\begin{split}
& \dot{W}=\bar{W}-\frac{1}{2}\sum_{i=1}^{N}\sum_{j \in N_i}\beta_{ij} (y_i-y_j) h(y_i- y_j)\\
 &+\frac{1}{N}\sum_{i=1}^{N}\sum_{j=1}^{N}(y_i-y_j)\phi_i +\frac{1}{2}\sum_{i=1}^{N}\sum_{j \in N_i}(\beta_{ij}-\bar{\beta})(y_i-y_j) h(y_i- y_j)\\
 &=\bar{W}+\frac{1}{2N}\sum_{i=1}^{N}\sum_{j=1}^{N}(y_i-y_j)(\phi_i-\phi_j)\\
 &-\frac{\bar{\beta}}{2}\sum_{i=1}^{N}\sum_{j \in N_i}(y_i-y_j) h(y_i- y_j)\\
 &\leq \bar{W}+\frac{1}{2N}\sum_{i=1}^{N}\sum_{j=1}^{N}\norm{y_i-y_j}_2\norm{\phi_i-\phi_j}_2\\
 &-\frac{\bar{\beta}}{2} \sum_{i =1}^N \sum_{j \in N_i} \frac{\norm{y_i-y_j}_2^2}{\norm{y_i-y_j}_2+\epsilon e^{-c t}}\\
 &\leq \bar{W}+ \frac{(N-1)\bar{\phi}}{4}\sum_{i=1}^{N}\sum_{j \in N_i}\norm{y_i-y_j}_2\\
 &-\frac{\bar{\beta}}{2} \sum_{i =1}^N \sum_{j \in N_i} \frac{\norm{y_i-y_j}_2^2}{\norm{y_i-y_j}_2+\epsilon e^{-c t}},
\end{split}
\normalsize\end{equation*} \normalsize
where in the last inequality Assumption \ref{as-bound-double} is used. Selecting a $\bar{\beta}$ such that $\bar{\beta}\geq \frac{ (N-1)\bar{\phi}}{2}$, we obtain
\begin{equation} \label{dww} \small 
\begin{split}
\dot{W} <\bar{W} +\frac{\bar{\beta}}{2} \sum_{i =1}^N \sum_{j \in N_i} \epsilon e^{-c t}.
\end{split}
\end{equation}\normalsize
If we can show $\bar{W}<-\psi W$ (or equivalently $\bar{W}+\psi W<0$), then knowing that $e^{-c t} \to 0$ as $t \to \infty$, Lemma 2.19 in \cite{Qu}
implies that the system \eqref{errordbldissgnvar} is asymptotically stable. Note that $\bar{W}+\psi W=$
\begin{equation}  \label{barwpsiw}
\left( {\begin{array}{cc} (-\gamma\mu+ \alpha \gamma \psi+\zeta \mu \psi) L-2\psi^2 \gamma I_{N}  & 0 \\ 0& (\gamma+ \zeta \psi)I_{N}-\alpha \zeta L \\ \end{array} } \right).
\end{equation}
Applying Lemma \ref{schur}, we obtain that \eqref{barwpsiw} is negative definite if
\begin{equation}  \label{condneg1}
\begin{split}
& (-\gamma\mu+ \alpha \gamma \psi+\zeta \mu \psi) L-2\psi^2 \gamma I_{N}<0 \\
& (\gamma+ \zeta \psi)I_{N}-\alpha \zeta L<0.
\end{split}
\end{equation} \normalsize 
To satisfy the first condition in \eqref{condneg1}, we just need to show $-\gamma\mu+ \alpha \gamma \psi+\zeta \mu \psi<0$. 
Using conditions \eqref{con1} and \eqref{con4}, we have $-\gamma\mu+ \alpha \gamma \psi+\zeta \mu \psi< -\gamma\mu+ \frac{\mu \gamma}{2} +\zeta \mu \psi< -\gamma\mu+ \frac{\mu \gamma}{2} +\frac{\mu \gamma}{2} <0$.
To satisfy the second condition in \eqref{condneg1}, we have $(\gamma+ \zeta \psi)I_{N}-\alpha \zeta L<(\gamma+ \zeta \psi)I_{N}-\alpha \zeta \lambda_2[L]I_{N}<0$, 
where Lemma \ref{eigvalue} and condition \eqref{con3} are employed, respectively. Hence, $\bar{W}<-\psi W $ holds and the agents reach consensus as $t\rightarrow \infty$. Now, similar to the proof of Theorem \ref{theorem1}, if $H_i(x_i,t)=H_j(x_j,t), \forall t$ and $\forall i,j \in {\mathcal I}$, it can be shown that $\sum_{j=1}^N \nabla f_j(x_j,t)$ will converge to zero asymptomatically. Now, under Assumption \ref{as-exist} and the assumption that $\sum_{i=1}^N f_i(x,t)$ is convex, Lemma \ref{lem-intro-gradzero} is employed. Using the fact that $x_i(t)\to x_j(t), \forall i,j \in {\mathcal I}$ as $t \to \infty$, it is easy to see that the optimization goal \eqref{goal} is achieved.
\end{proof}
\begin{remark}
It is worth mentioning that if we have $\frac{\gamma}{\alpha \zeta}<\lambda_2[L]$, there always exists a positive coefficient $\psi$ such that conditions \eqref{con3}-\eqref{con4} hold. However, selecting $\psi$ based on conditions \eqref{con3}-\eqref{con4} affects the convergence speed, where by having a larger $\psi$ the agents reach consensus faster. To satisfy conditions \eqref{con3} and \eqref{con4}, it is sufficient to have $\frac{2\lambda_2[L] \alpha}{3}> \frac{\gamma}{2\zeta}> \psi $ (e.g, selecting a large $\alpha$ and choosing proper $\gamma, \zeta$ and $\psi$). It can also be seen that selecting a large enough $\mu$, \eqref{con1} can be satisfied.
\end{remark}

\begin{remark} 
The results in Appendix \ref{AppendixA} can be modified for Theorem \ref{theoremsgnvar}, where we can show that Assumption \ref{as-bound-double} holds under the same conditions mentioned in Remark \ref{Remarkbounddouble}. 
\end{remark}
\subsection{Distributed Time-Varying Convex Optimization Using Time-Invariant Approximation of Signum Function} \label{subsecdblsgninv}
In this subsection, our focus is on replacing the signum function, with a time-invariant approximation
\begin{equation} \label{hinv} \small h(z)=\frac{z}{\norm{z}_2+\epsilon},
\end{equation}\normalsize
where $\epsilon>0$. Here, the boundary layer $\epsilon$ is constant. Employing \eqref{hinv}, instead of \eqref{hvaring} in the control algorithm \eqref{udbldissgnvary} makes the controller easier to implement in real applications. The trade-off is that the agents will no longer reach consensus, which introduces additional complexities in convergence analysis. Establishing analysis on the optimization error bound in this case is a nontrivial task, which is introduced in this subsection.
The reason that the time-invariant continuous approximation \eqref{hinv} cannot ensure distributed optimization with zero error is that the global optimal trajectory is not even an equilibrium point of the closed-loop system whenever a time-invariant continuous approximation is introduced. It is worthwhile to mention that if the signum function were replaced with a different time-invariant continuous approximation other than \eqref{hinv}, there would be no guarantee that the same conclusion in this subsection would still hold and further careful analysis would be needed.

\begin{theorem} \label{theoremsgninv}
Suppose that the graph ${\mathcal G}$ is connected, Assumptions \ref{as-exist}, \ref{as1each} and \ref{as-bound-double} hold and the gradients of the cost functions can be written as $\nabla f_i(x_i,t)=\sigma x_i+g_i(t), \ \forall i \in {\mathcal I}$, where $\sigma$ and $g_i(t)$ are, respectively, a positive coefficient and a time-varying function. If conditions \eqref{con3}-\eqref{con4} hold, using (\ref{udbldissgnvary}) with $h(\cdot)$ given by \eqref{hinv} for (\ref{double}), we have
\begin{equation} \label{xvbound}   
\lim_{t \to \infty}[\frac{1}{N}\sum_{i=1}^N x_i - x^*]=0, \quad \lim_{t \to \infty}[\frac{1}{N}\sum_{i=1}^N v_i - v^*]=0,
\end{equation} \normalsize
where $x^*$ and $v^*$ are the position and the velocity of the optimal trajectory, respectively. In addition, the agents track the optimal trajectory with bounded errors such that as $t \to \infty$
\begin{equation} \label{resulttheosgninv} \small 
\begin{split} 
&\norm{x_i-x^*}_2  <\sqrt{\frac{\bar{\phi} N(N-1)^2\epsilon}{4\psi \lambda_\text{min}[P]}},\\
&\norm{ v_i-v^*}_2  <\sqrt{\frac{\bar{\phi} N(N-1)^2\epsilon}{4\psi \lambda_\text{min}[P]}}, \ \ \forall i \in {\mathcal I},
\end{split} \end{equation} \normalsize where $P$ is defined after \eqref{lypapWdbl}.
\end{theorem}

\begin{proof}
The proof will be separated into two parts. In the first part, we show that the consensus error will remain bounded. Then in the second part, we show that the error between the agents' states and the optimal trajectory will remain bounded. 

Define the Lyapunov function candidate $W$ as in \eqref{lypapWdbl}. Similar to the proof in Theorem \ref{theoremsgnvar}, with $h(\cdot)$ given by \eqref{hinv} instead of \eqref{hvaring}, we obtain $\dot{W} <\bar{W} +\frac{\bar{\beta}}{2} \sum_{i =1}^N \sum_{j \in N_i} \epsilon< -\psi W+\frac{\bar{\beta}}{2}N(N-1)\epsilon,$ where $\bar{\beta}$ is selected such that $\bar{\beta}\geq \frac{ (N-1)\bar{\phi}}{2}$. Then we have $0 \leq W(t)\leq \frac{\bar{\beta} N(N-1)\epsilon}{2\psi} (1-e^{-\psi t})+ W(0)e^{-\psi t}$.
Therefore, as $t\rightarrow \infty$, we have 
\begin{equation*}  \begin{split} \norm{\left( {\begin{array}{cc} e_X \\ e_V \\ \end{array} } \right)}_2^2 \lambda_{\text{min}} [P] &\leq W=  \left( {\begin{array}{cc} e_X \\ e_V \\ \end{array} } \right)^T P \left( {\begin{array}{cc} e_X \\ e_V \\ \end{array} } \right)\\ &\leq \frac{\bar{\beta} N(N-1)\epsilon}{2\psi}. \end{split} \end{equation*} \normalsize
Now, it can be seen that there exists a bound on the position and velocity consensus errors as $t \to \infty$, that is,
\begin{equation} \label{bounderrorsgninv}  \small 
\begin{split} 
&\norm{x_i - \frac{1}{N} \sum_{j=1}^N x_j}_2  <\sqrt{\frac{\bar{\phi} N(N-1)^2\epsilon}{4\psi \lambda_\text{min} [P]}},\\
&\norm{v_i - \frac{1}{N} \sum_{j=1}^N v_j}_2 <\sqrt{\frac{\bar{\phi} N(N-1)^2\epsilon}{4\psi \lambda_\text{min} [P]}},
\end{split} \end{equation} \normalsize
where it is easy to see that by selecting larger $\psi$ satisfying conditions \eqref{con3}-\eqref{con4}, the error bound will be smaller.
 
In what follows, we focus on finding the relation between the optimal trajectory and the agents' states. According to Assumption \ref{as-exist} and using Lemma \ref{lem-intro-gradzero}, we know $\sum_{j=1}^N \nabla f_j(x^*,t)=0$. Hence, under the assumption of $\nabla f_i(x_i,t)=\sigma x_i+g_i(t)$, the optimal trajectory is
\begin{equation} \label{xvoptimal}  \small
\begin{split} 
& x^*=\frac{-\sum_{j=1}^N g_j}{N \sigma}, \hspace*{.5cm} v^*=\frac{-\sum_{j=1}^N \dot{g}_j}{N \sigma}.
\end{split} \end{equation} \normalsize
Similar to the proof of Theorem \ref{theorem1}, we can show that, regardless of whether consensus is reached or not, it is guaranteed that $\sum_{j=1}^N \nabla f_j(x_j,t)$ will converge to zero asymptomatically. As a result, we have $ \sum_{j=1}^N x_i \to \frac{-\sum_{j=1}^N g_j}{ \sigma}$ and $\sum_{j=1}^N v_i \to \frac{-\sum_{j=1}^N \dot{g}_j}{\sigma}$. By using \eqref{xvoptimal}, we can conclude \eqref{xvbound}. According to \eqref{bounderrorsgninv}, it follows that \eqref{resulttheosgninv} holds.
\end{proof}
\begin{remark} \label{comparevarandinv}
Using the invariant approximation of the signum function \eqref{hinv}, instead of the time-varying one \eqref{hvaring}, makes the implementation easier in real applications. However, the results show that the team cost function is not exactly minimized and the agents track the optimal trajectory with a bounded error. It also restricts the acceptable cost functions to a class that takes in the form $\nabla f_i(x_i,t)=\sigma x_i+g_i(t)$. 
\end{remark}

\begin{remark} 
The results in Appendix \ref{AppendixA} can be modified for Theorem \ref{theoremsgnvar}, where under the assumption of $\nabla f_i(x_i,t)=\sigma x_i+g_i(t)$, it is easy to show that Assumption \ref{as-bound-double} holds, if $\norm{g_i(t) -g_j(t)}_2, \ \norm{\dot{g}_i(t) -\dot{g}_j(t)}_2$ and $\norm{\ddot{g}_i(t) -\ddot{g}_j(t)}_2, \forall t$ and $\forall i,j \in {\mathcal I},$ are bounded.
\end{remark}

\begin{remark} \label{directgraph}
The algorithms introduced Subsections \ref{subsecsingledis1}, \ref{secsingledis2}, \ref{subsecdbldis1}, \ref{subsecdblsgntimevar} and \ref{subsecdblsgninv} are still valid in the case of strongly connected weight balanced directed graph ${\mathcal G}$. In our proofs, $L$ can be replaced with symmetric matrix $L+L^T$ as $x^T L x=\frac{1}{2} x^T (L+L^T) x$. Since ${\mathcal G}$ is strongly connected weight balanced, $L+L^T$ is positive semidefinite with a simple zero eigenvalue. Note that applying the introduced algorithms for directed graphs, we need to redefine $\lambda_2$ as the smallest nonzero eigenvalue of $L+L^T$.
\end{remark}

\section{Distributed Time-Varying Convex Optimization with Swarm tracking behavior} \label{sec:swarm}
In this section, we introduce two distributed optimization algorithms with swarm tracking behavior, where the center of the agents tracks the optimal trajectory defined by \eqref{costdis} for single-integrator and double-integrator dynamics while the agents avoid inter-agent collisions and maintain connectivity.

\subsection{Distributed Convex Optimization with Swarm Tracking behavior for Single-Integrator Dynamics} \label{subsecswarmsingle}
In this subsection, we focus on the distributed optimization problem with swarm tracking behavior for single-integrator dynamics \eqref{single}. To solve this problem, we propose the algorithm 
\begin{equation} \label{usingleswarm}
\begin{split}
u_i(t) =& - \beta \text{sgn}( \sum_{j \in N_i(t)} \frac{\partial V_{ij}}{\partial x_i})+\phi_i, 
\end{split}
\end{equation}
where $V_{ij}$ is a potential function between agents $i$ and $j$ to be designed, $\beta$ is positive, and $\phi_i$ is defined in \eqref{usingledis1}. In \eqref{usingleswarm}, each agent uses its own position and the relative positions between itself and its neighbors. It is worth mentioning that in this subsection, we assume each agent has a communication/sensing radius $R$, where if $\norm{x_i-x_j}_2<R$ agent $i$ and $j$ become neighbors. Our proposed algorithm guarantees connectivity maintenance in the sense that if the graph ${\mathcal G}(0)$ is connected, then for all $t, {\mathcal G}(t)$ will remain connected. Before our main result in this subsection, we need to define the potential function $V_{ij}$.

\begin{definition} \label{defswarm} \cite{caoren12}
The potential function $V_{ij}$ is a differentiable nonnegative function of $\norm{x_i-x_j}_2$, which satisfies the following conditions
\begin{itemize}
\item[1)] $V_{ij}=V_{ji}$ has a unique minimum in $\norm{x_i-x_j}_2=d_{ij}$, where $d_{ij}$ is the desired distance between agents $i$ and $j$ and $R> \max_{i,j} d_{ij}$.
\item[2)] $V_{ij} \rightarrow \infty$ if $\norm{x_i-x_j}_2 \rightarrow 0$.
\item [3)] $V_{ii}=c$, where $c$ is a constant. \vspace*{0.1cm}
\item [4)] \small $\begin{cases} \frac{\partial V_{ij}}{\partial(\norm{x_i-x_j}_2)} =0 &  \norm{x_i(0)-x_j(0)}_2\geq R,\ \norm{x_i-x_j}_2\geq R,
\\ \frac{\partial V_{ij}}{\partial(\norm{x_i-x_j}_2)} \to \infty & \norm{x_i(0)-x_j(0)}_2< R, \ \norm{x_i-x_j}_2 \to R.
\end{cases}$ \normalsize
\end{itemize}
\end{definition}\vspace*{0.1cm}

The motivation of the last condition in Definition 5.1 is to maintain the initially existing connectivity patterns. It guarantees that two agents which are neighbors at $t=0$ remain neighbors. However, if two agents are not neighbors at $t=0$, they do not need to be neighbors at $t>0$ (see \cite{caoren12}).

\begin{theorem} \label{theoremswarm}
Suppose that graph ${\mathcal G}(0)$ is connected, Assumptions \ref{as-exist} and \ref{as1each} hold and the gradient of the cost functions can be written as $\nabla f_i(x_i,t)=\sigma x_i+g_i(t), \ \forall i \in {\mathcal I}$, where $\sigma$ and $g_i(t)$ are,  respectively, a positive coefficient and a time-varying function. If $\beta > \norm{\phi_i}_1, \ \forall i \in {\mathcal I}$, for the system \eqref{single} with the algorithm \eqref{usingleswarm}, the center of the agents tracks the optimal trajectory while the agents maintain connectivity and avoid inter-agent collisions.
\end{theorem}
\begin{proof}
Define the positive semidefinite Lyapunov function candidate
 \begin{equation} \label{Lyapswarm}
W= \frac{1}{2}\sum_{i=1}^N \sum_{j=1}^N V_{ij}.
\end{equation}
The time derivative of $W$ is obtained as $\dot{W}= \frac{1}{2}\sum_{i=1}^N \sum_{j=1}^N \big(\frac{\partial V_{ij}}{\partial x_i} \dot{x}_i + \frac{\partial V_{ij}}{\partial x_j} \dot{x}_j \big)= \sum_{i=1}^N \sum_{j=1}^N \frac{\partial V_{ij}}{\partial x_i} \dot{x}_i$,
where in the second equality, Lemma 3.1 in \cite{caoren12} has been used. Now, rewriting $\dot{W}$ along the closed-loop system \eqref{usingleswarm} and \eqref{single}, we have 
\begin{equation*} \small 
\begin{split}
&\dot{W}= - \beta \sum_{i=1}^N \sum_{j=1}^N \frac{\partial V_{ij}}{\partial x_i} \text{sgn}( \sum_{j=1}^N \frac{\partial V_{ij}}{\partial x_i}) \\
&+\sum_{i=1}^N  \sum_{j=1}^N \frac{\partial V_{ij}}{\partial x_i} \phi_i \leq  \sum_{i=1}^N \bigg(\norm{ \sum_{j=1}^N \frac{\partial V_{ij}}{\partial x_i}}_1 [\norm{\phi_i}_1 -\beta] \bigg).
\end{split}
\end{equation*}\normalsize
It is easy to see that if $\beta> \norm{\phi_i}_1, \ \forall i \in {\mathcal I},$ then $\dot{W}$ is negative semidefinite. Therefore, having $W\geq 0$ and $\dot{W} \leq 0$, we can conclude that $V_{ij} \in \mathcal{L}_{\infty}$. Since $V_{ij}$ is bounded, based on Definition \ref{defswarm}, it is guaranteed that there will be no inter-agent collision and the connectivity is maintained.

In what follows, we focus on finding the relation between the optimal trajectory and the agents' positions. Based on Definition \ref{defswarm}, we can obtain that
\begin{equation} \label{rondvrelation}\frac{\partial V_{ij}}{\partial x_i}=\frac{\partial V_{ji}}{\partial x_i}=-\frac{\partial V_{ij}}{\partial x_j}. \end{equation}
Now, by summing both sides of the closed-loop system \eqref{single} with the control algorithm \eqref{usingleswarm}, for $i \in \mathcal{I}$ we have $\sum_{j=1}^N \dot{x}_j = \sum_{j=1}^N \phi_j$. We also know that the agents have identical Hessians since it is assumed that $\nabla f_i(x_i,t)=\sigma x_i+g_i(t)$.
Now, similar to the proof of Theorem \ref{theorem1}, regardless of whether consensus is reached or not, we can show that $\sum_{j=1}^N \nabla f_j(x_j,t)$ will converge to zero asymptomatically. Hence we have $\sum_{j=1}^N x_i \to \frac{-\sum_{j=1}^N g_j}{ \sigma}$. On the other hand, using Lemma \ref{lem-intro-gradzero} and under Assumption \ref{as-exist}, we know $\sum_{j=1}^N \nabla f_j(x^*,t)=0$. Hence, the optimal trajectory is $x^*=\frac{-\sum_{j=1}^N g_j}{N \sigma}$. This implies that $\frac{1}{N}\sum_{j=1}^N x_i \to x^*$, where we have shown that the center of the agents will track the team cost function minimizer.
\end{proof}

\begin{remark}
In Appendix \ref{AppendixbB}, it is shown that a constant $\beta$ can be selected such that $\beta > \norm{\phi_i}_1,  \forall t$ and $\forall i \in \mathcal{I}$, if $\norm{g_i(t)}_2$ and $\norm{\dot{g}_i(t)}_2, \forall t$ and $\forall i \in \mathcal{I}$, are bounded.
In particular, it is shown that such a constant $\beta$ can be determined at time $t=0$ by using the agents' initial states and the upper bounds on $\norm{g_i(t)}_2$ and $\norm{\dot{g}_i(t)}_2, \forall t$ and $\forall i \in \mathcal{I}$.
\end{remark}

\subsection{Distributed Convex Optimization with Swarm Tracking Behavior for Double-Integrator Dynamics} \label{subsecswarm}
In this subsection, we focus on distributed time-varying optimization with swarm tracking behavior for double-integrator dynamics \eqref{double}. We will propose an algorithm, where each agent has access to only its own position and the relative positions and velocities between itself and its neighbors. We propose the algorithm 
\begin{equation} \label{uswarm}
\begin{split}
u_i(t) =& - \sum_{j \in N_i(t)} \frac{\partial V_{ij}}{\partial x_i} -\beta  \sum_{j \in N_i(t)} \text{sgn}(v_i -v_j)+\phi_i, 
\end{split}
\end{equation}
where $V_{ij}$ is defined in Definition \ref{defswarm}, $\beta$ is a positive coefficient, and $\phi_i$ is defined in \eqref{udbldis1}.

\begin{theorem} \label{theoremswarm-double}
Suppose that the graph ${\mathcal G}(0)$ is connected, Assumptions \ref{as-exist} and \ref{as1each} hold and the gradient of the cost functions can be written as $\nabla f_i(x_i,t)=\sigma x_i+g_i(t), \ \forall i \in {\mathcal I}$. If $\beta > \frac{ \norm{(\Pi \otimes I_m)\Phi}_2}{\sqrt{\lambda_{2}[L(t)]}}$ holds, for the system (\ref{double}) with the algorithm \eqref{uswarm}, the center of the agents tracks the optimal trajectory, the agents' velocities track the optimal velocity, and the agents maintain connectivity and avoid inter-agent collisions.
\end{theorem}

\begin{proof}
Writing the closed-loop system \eqref{double} with the control algorithm \eqref{uswarm} based on the consensus errors $e_X$ and $e_V$ defined in \eqref{exev}, we have
\begin{equation} \label{errorsysswarm}
\begin{cases} \dot{e}_X =& e_V \\ 
\dot{e}_V =&-\alpha(L(t)\otimes I_m) e_V-\beta (D(t) \otimes I_m) \text{sgn}\big([D^T(t)\otimes I_m]e_V\big)
\\& \begin{pmatrix}
  \sum_{j \in N_1} \frac{\partial V_{1j}}{\partial e_{X_1}} \\
    \vdots \\
   \sum_{j \in N_N} \frac{\partial V_{Nj}}{\partial e_{X_N}}
 \end{pmatrix} +(\Pi\otimes I_m) \Phi.
\end{cases} 
\end{equation}
Define the positive semidefinite Lyapunov function candidate $W= \frac{1}{N}\sum_{i=1}^N \sum_{j=1}^N V_{ij} +\frac{1}{2} e_V^T e_V$. The time derivative of $W$ along \eqref{errorsysswarm} can be obtained as $\dot{W}= \frac{1}{2}\sum_{i=1}^N \sum_{j=1}^N \big(\frac{\partial V_{ij}}{\partial e_{X_i}} e_{V_i} + \frac{\partial V_{ij}}{\partial e_{X_j}} e_{V_j} \big)+ e_V^T \dot{e}_V$.
Using Lemma 3.1 in \cite{caoren12}, $\dot{W}$ can be rewritten as 
\begin{equation*} \small 
\begin{split}
\dot{W}=&  -\alpha e_V^T (L(t)\otimes I_m) e_V -\beta e_V^T (D(t) \otimes I_m) \text{sgn}\big([D^T(t)\otimes I_m]e_V\big)\\&+ e_V^T (\Pi\otimes I_m) \Phi.
\end{split}
\end{equation*}\normalsize
Using a similar argument to that in \eqref{local-lambda2}, we obtain that if $\beta \sqrt{\lambda_{2}[L(t)]} > \norm{(\Pi \otimes I_m)\Phi}_2,$ then $\dot{W}$ is negative semidefinite. Therefore, having $W\geq 0$ and $\dot{W} \leq 0$, we can conclude that $V_{ij}, e_v \in \mathcal{L}_{\infty}$. By integrating both sides of $\dot{W} \leq -\alpha e_V^T (L(t)\otimes I_m) e_V$, we can see that  $e_v \in \mathcal{L}_{2}$. Now, applying Barbalat's Lemma \cite{sheida22j}, we obtain that $e_V$ converges to zero asymptotically, which means that the agents' velocities reach consensus as $t\rightarrow \infty$. Since $V_{ij}$ is bounded, it is guaranteed that there will be no inter-agent collision and the connectivity is maintained. 

In the next step, using \eqref{rondvrelation}, by summing both sides of the closed-loop system \eqref{double} with the control algorithm \eqref{uswarm} for $i \in \mathcal{I},$ we have $\sum_{j=1}^N \dot{v}_j = \sum_{j=1}^N \phi_j$. Now, if the team cost function $\sum_{i=1}^N f_i(x,t)$ is convex and $\nabla f_i(x_i,t)=\sigma x_i+g_i(t)$, applying a procedure similar to the proof of Theorem \ref{theoremsgninv}, we can show that $\sum_{j=1}^N \nabla f_j(x_j,t)$ will converge to zero asymptomatically and \eqref{xvbound} holds. Particularly, we have shown that the average of agents' states tracks the optimal trajectory. Because the agents' velocities reach consensus as $t \to \infty$, we have that $v_i$ approaches $v^*$ as $t \to \infty$. 
\end{proof}

\begin{remark} \label{remarkswarmdouble}
The assumption $\beta > \frac{ \norm{(\Pi \otimes I_m)\Phi}_2}{\sqrt{\lambda_{2}[L(t)]}}$ can be interpreted as a bound on the difference between the agents' internal signals. Using the fact that  $\lambda_{2}[L(t)]$ is lower bounded above 0, there always exists a $\beta$ satisfying $\beta > \frac{ \norm{(\Pi \otimes I_m)\Phi}_2}{\sqrt{\lambda_{2}[L(t)]}}$ if Assumption \ref{as-bound-double} holds. Here, to satisfy Assumption \ref{as-bound-double}, with $\nabla f_i(x_i,t)=\sigma x_i+g_i(t), \ \forall i \in {\mathcal I}$, the boundlessness of $\norm{g_i(t)-g_j(t)}_2, \norm{\dot{g}_i(t)-\dot{g}_j(t)}_2$, and $\norm{\ddot{g}_i(t)-\ddot{g}_j(t)}_2, \forall t$ and $\forall i,j \in {\mathcal I},$ is sufficient.
\end{remark}
\begin{remark} \label{switching-graph}
The algorithms proposed in Subsections \ref{secsingledis2} and \ref{subsecdblestim} and Section \ref{sec:swarm} are provided for time-varying graphs. For algorithms introduced in Subsections \ref{subsecsingledis1}, \ref{subsecdbldis1}, \ref{subsecdblsgntimevar} and \ref{subsecdblsgninv}, the current results are demonstrated for static graphs. However, the results are valid for time-varying graphs if the graph ${\mathcal G}(t)$ is connected for all $t$ and a sufficiently large constant gain is used, instead of the time-varying adaptive gains. In particular, the constant gain $\beta$ should satisfy $\beta> \frac{\norm{(\Pi \otimes I_m)\Phi}_2}{\sqrt{\lambda_2[L(t)]}}$. To select such a $\beta$, we need to know $\bar{\phi}$, defined in Assumption \ref{as-bound-single} (or Assumption \ref{as-bound-double} in the case of double-integrator dynamics), and $\beta_x$ (and $\beta_v$ in case of double-integrator dynamics), defined in Appendix \ref{AppendixA}.
\end{remark}
\begin{remark}
All the proposed algorithms are also applicable to non-convex functions. However, in this case it is just guaranteed that the agents converge to a local optimal trajectory of the team cost function. 
\end{remark}
\section{Simulation and Discussion } \label{sec:sim}
In this section, we present various simulations to illustrate the theoretical results in previous sections. Consider a team of six agents. The interaction among the agents is described by an undirected ring graph. The agents' goal is to minimize the team cost function $\sum_{i=1}^6 f_i (x_i,t)$, where $x_i=({r_{x_i}},{r_{y_i}})^T$ is the coordinate of agent $i$ in $2D$ plane, subject to $x_i=x_j$. 

In our first example, we apply the algorithm \eqref{usingledis1} for single-integrator dynamics \eqref{single}. The local cost function for agent $i$ is chosen as 
\begin{equation} \label{costfunc1}
f_i(x_i,t)= (r_{x_i}- i \text{sin}(t))^2+ (r_{y_i}- i \text{cos}(t))^2,
\end{equation}
where it is easy to see that the optimal point of the team cost function creates a trajectory of a circle whose center is at the origin and radius is equal to $\frac{21}{6}$. For \eqref{costfunc1}, Assumption \ref{as1each} and the conditions for agents' cost functions in Remark \ref{Remarkboundsingle} hold. In addition they have identical Hessians and the team cost function is convex. $\beta_{ij}(0)=\beta_{ji}(0)$ are chosen randomly within (0.1,2) in the algorithm \eqref{usingledis1}. The trajectory of the agents and the optimal trajectory is shown in Fig. \ref{sgldisfig}. It can be seen that the agents reach consensus and track the optimal trajectory which minimizes the team cost function.
\begin{figure}[t]
\begin{center} \hspace*{-0.45cm}
\vspace{0.1cm} {\scalebox{0.29}{\includegraphics*{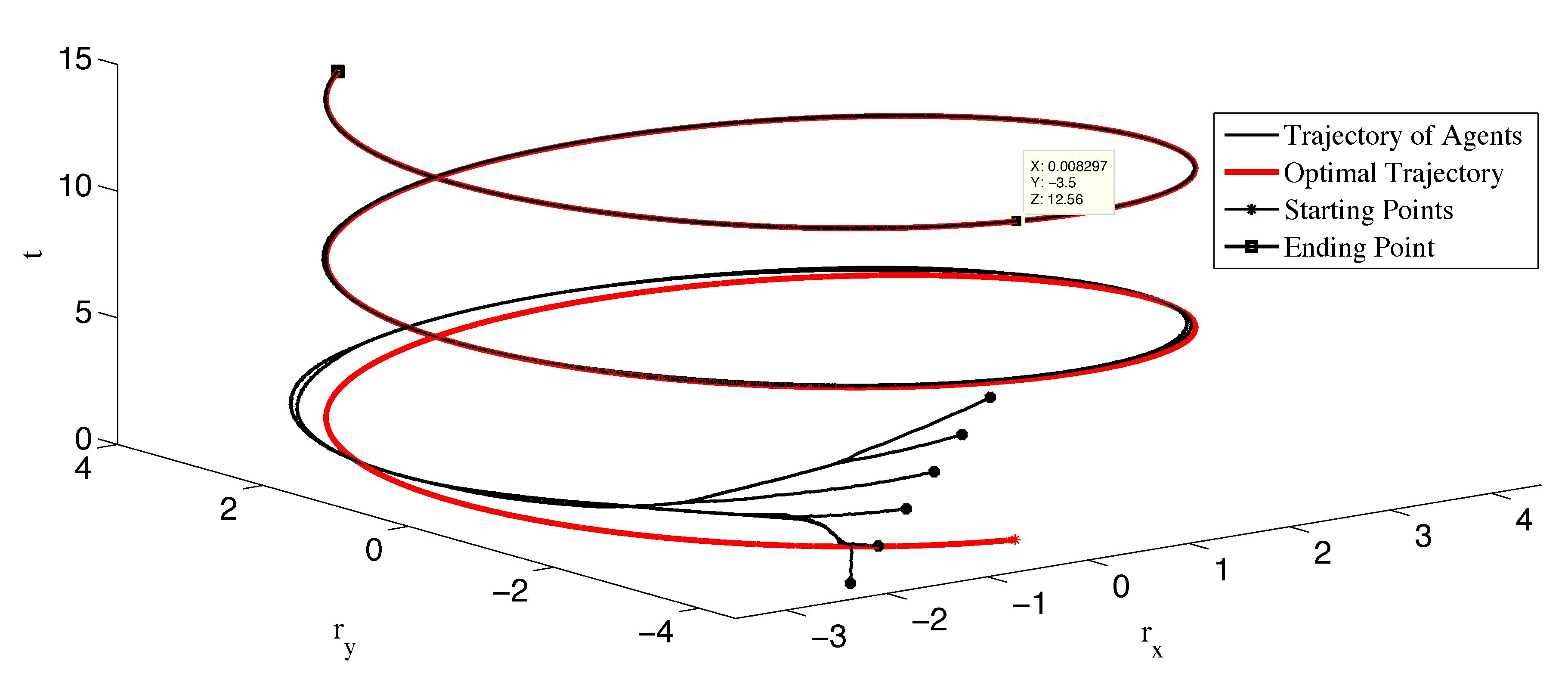}}}
\vspace{-.2cm} \caption{Trajectories of all agents along with the optimal trajectory using the algorithm \eqref{usingledis1} for the local cost functions \eqref{costfunc1}} \vspace{-0.1cm}
\label{sgldisfig}
\end{center}
\end{figure}

In the case of double-integrator dynamics, we first give an example to illustrate the algorithm \eqref{udbldis1} for \eqref{double} with the local cost functions defined by \eqref{costfunc1}. Choosing the coefficients in \eqref{udbldis1} as $\mu =5, \alpha=12, \gamma=5, \zeta=12$, and $\beta_{ij}(0)=\beta_{ji}(0)$ randomly within (0.1,2), the agents reach consensus and the team cost function is minimized as shown in Fig. \ref{dbldisfig1}.

\begin{figure}[t] 
\begin{center} \hspace*{-.4cm}
\vspace{0.1cm} {\scalebox{0.28}{\includegraphics*{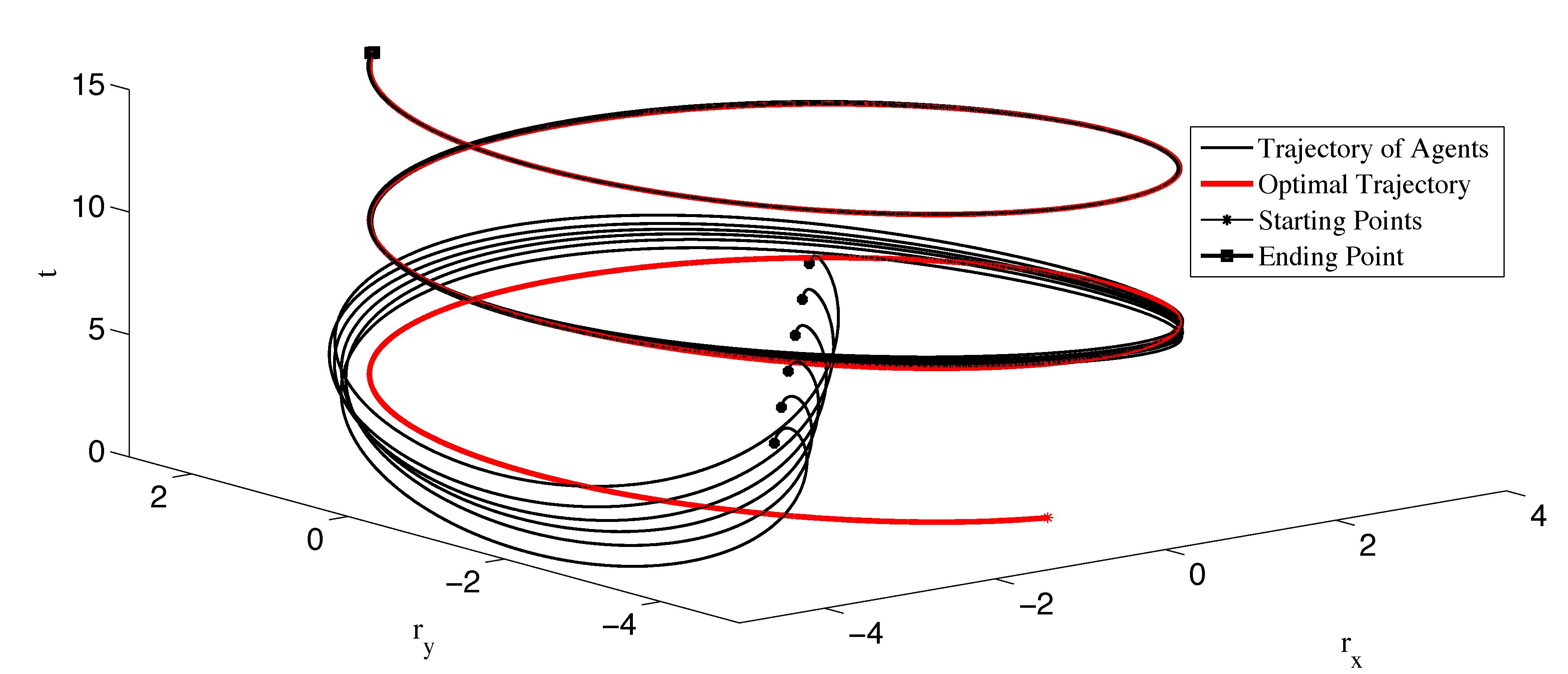}}}
\vspace{-0.3cm} \caption{Trajectories of all agents along with the optimal trajectory using the algorithm \eqref{udbldis1} for the local cost functions \eqref{costfunc1}} \vspace{-0.3cm}
\label{dbldisfig1}\end{center}
\end{figure}

In our next example, we illustrate the results obtained in Subsection \ref{subsecdblestim}, where it has been clarified that the algorithm \eqref{koldbl1} -\eqref{contestdbl} can be used for local cost functions with nonidentical Hessians. Here, the local cost functions 
\begin{equation} \label{costfunc2}
f_i(x_i,t)= (\frac{r_{x_i}}{i} - \text{sin}(t))^2+ (\frac{r_{y_i}}{i} - \text{cos}(t))^2,
\end{equation}
will be used, where $H_i(x_i,t)=\frac{2}{i^2} I_2, \ \forall i \in {\mathcal I}$. It can be obtained that the team cost function's optimal trajectory creates a circle whose center is at the origin and radius is equal to $1.64$. The algorithm \eqref{koldbl1} -\eqref{contestdbl} with $\kappa=12, \rho=2, \alpha_1=0.1 $, and $\alpha_2=\frac{0.2}{1.1}$ is used for the system \eqref{double}. Fig. \ref{dblestimfig} shows that the team cost function is minimized.

 \begin{figure}[t] 
\begin{center} \hspace*{-.5cm}
\vspace{0.1cm} {\scalebox{0.28}{\includegraphics*{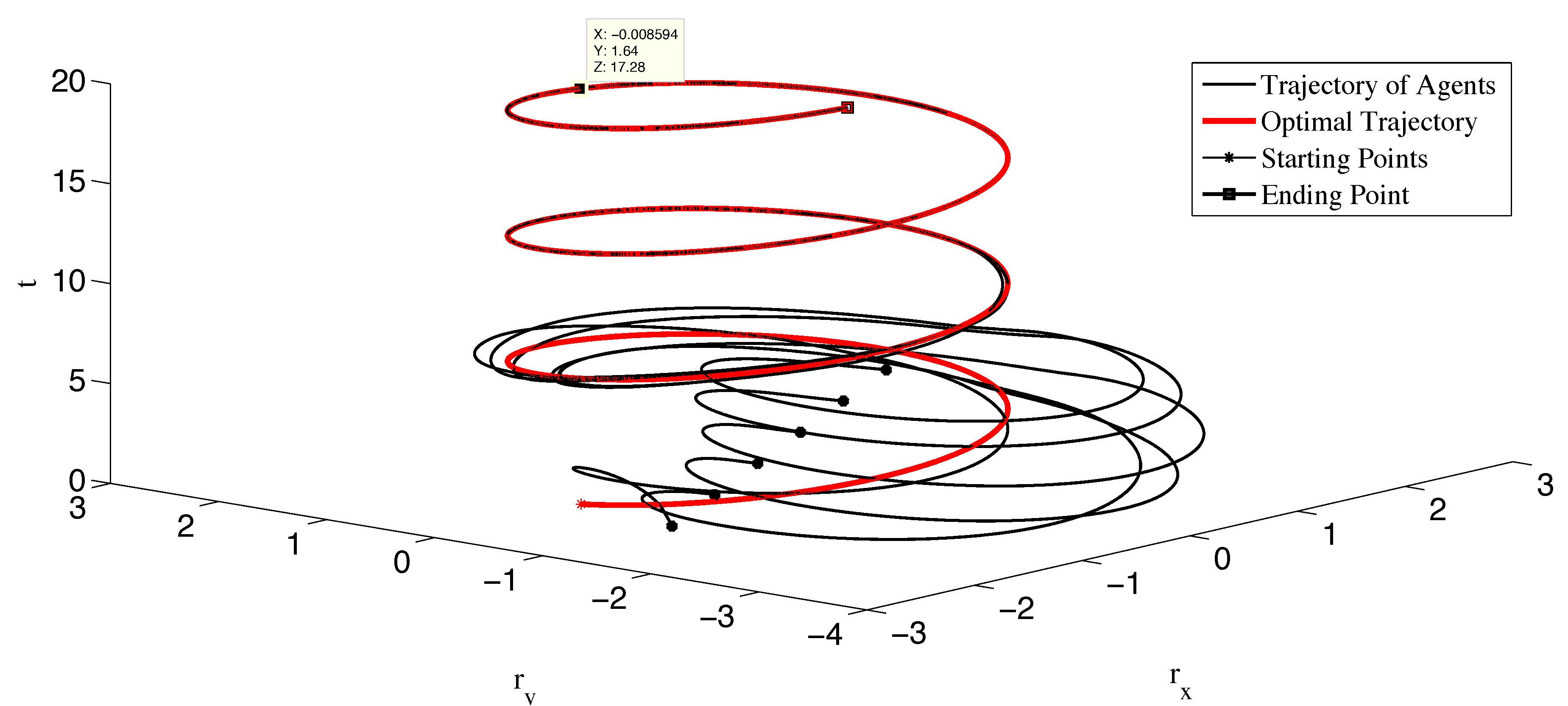}}}
\vspace{-0.2cm} \caption{Trajectories of all agents along with the optimal trajectory using the algorithm \eqref{koldbl1} -\eqref{contestdbl} for the local cost functions \eqref{costfunc2}} \vspace{-0.3cm}
\label{dblestimfig}\end{center}
\end{figure}

In our next example, the results in Subsection \ref{subsecdblsgninv} is illustrated, where the invariant approximation of the signum function is employed. Here, the algorithm \eqref{udbldissgnvary} with $h(\cdot)$ given by \eqref{hinv} is used to minimize the agents' team cost function for local cost functions defined as \eqref{costfunc1}. The coefficients are chosen as $\mu=5, \alpha=10, \gamma=5, \zeta=5, \epsilon=2$ and $\beta_{ij}(0)=\beta_{ji}(0)$ randomly within (0.1,2). Fig. \ref{dblsgninvfig} shows the agents' trajectories along with the optimal one. It is shown that the agents track the optimal trajectory with a bounded error.

\begin{figure}[t] 
\begin{center} \hspace*{-.3cm}
\vspace{0.1cm} {\scalebox{0.28}{\includegraphics*{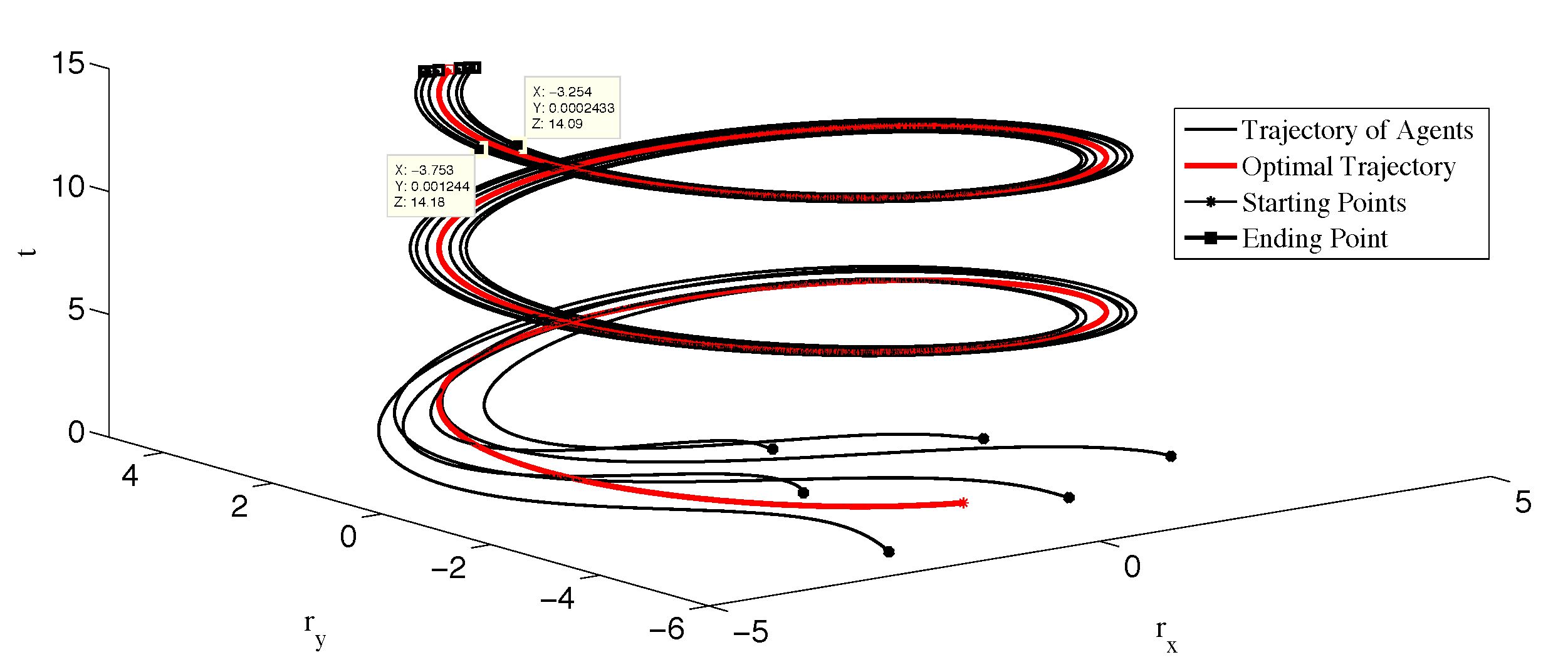}}}
\vspace{-.25cm} \caption{Trajectories of all agents using the algorithm \eqref{udbldissgnvary} and \eqref{hinv} for the local cost functions \eqref{costfunc1}} \vspace{-0.3cm}
\label{dblsgninvfig}\end{center}
\end{figure}

In our last illustration, the swarm tracking control algorithm \eqref{uswarm} is employed, where the local cost functions are defined as
\begin{equation} \label{costfunc3} \small
f_i(x_i,t)= (r_{x_i} +2i \frac{\text{sin}(0.5t)}{t+1})^2+ (r_{y_i} +i\text{sin}(0.1t))^2.
\end{equation} \normalsize
In this case, we let $R=5$. The parameter of \eqref{uswarm} is chosen as $\beta=20$. To guarantee the collision avoidance and connectivity maintenance, the potential function partial derivative is chosen as Eqs. (36) and (37) in  \cite{caoren12}, where $d_{ij}=0.5, \ \forall i,j$. Fig. \ref{dblswarmfig} shows that the center of the agents' positions tracks the optimal trajectory while the agents remain connected and avoid collisions.

\begin{figure}[t] 
\begin{center} \hspace*{-.4cm}
\vspace{0.1cm} {\scalebox{0.29}{\includegraphics*{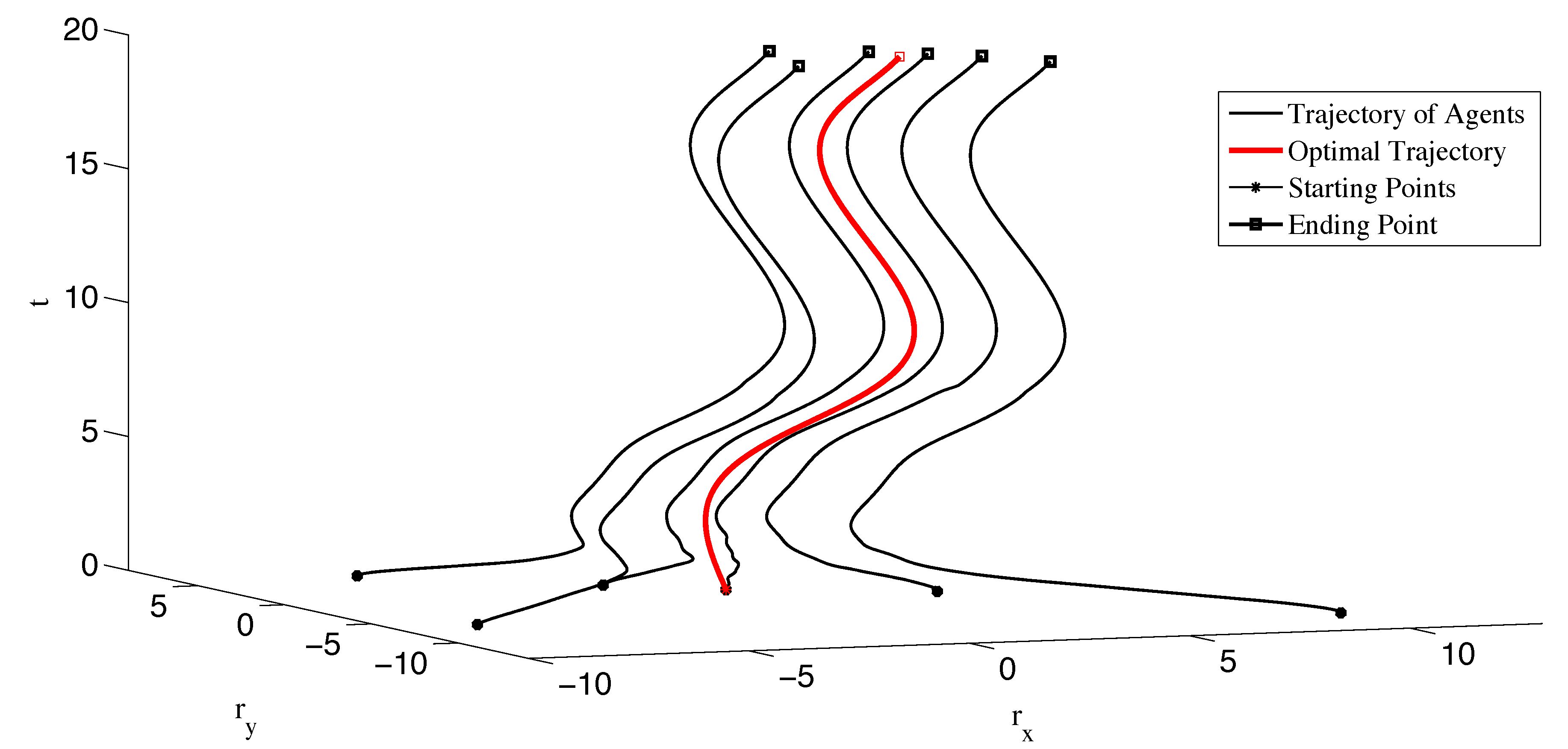}}}
\vspace{-0.28cm} \caption{Trajectories of all agents using the algorithm \eqref{uswarm} for local cost function \eqref{costfunc3}} \vspace{-0.3cm}
\label{dblswarmfig}\end{center}
\end{figure}


\section{Conclusions} \label{sec:conclusions}
In this paper, a time-varying distributed convex optimization problem was studied for continuous-time multi-agent systems, where the objective was to minimize the sum of the local time-varying cost functions.
Each local cost function was only known to an individual agent. Control algorithms have been designed for the cases of single-integrator and double-integrator dynamics. In both cases, as a first step, a centralized approach has been introduced to solve the optimization problem for convex time-varying cost functions. Then this problem has been solved in a distributed manner and a discontinuous algorithm with adaption gains has been proposed, where it was possible to rely on only local sensing.
To relax the restricted assumption imposed on the feasible cost functions, an estimator based algorithm has been proposed, where the agents used dynamic average tracking as a tool to estimate the centralized control input. However, the necessity of communication between neighbors was the drawback of the estimator based algorithm. Then in the case of double-integrator dynamics, we have focused on extending our proposed algorithms to improve them for real applications. Two continuous algorithms have been proposed which employed continuous approximations of the signum function. The first continuous algorithm used a time-varying approximation of the signum function, where we have shown that the team cost function was minimized and the agents reached consensus. In the second continuous algorithm, a time-invariant approximation of the signum function has been used which was easier to implement. The trade-off was that there existed a bounded error between the agents and the optimal trajectory. To add the inter-agent collision avoidance capability into our algorithms, two distributed convex optimization algorithms with swarm tracking behavior have been proposed for single-integrator and double-integrator dynamics. It has been shown that for both cases, the connectivity of the agents was maintained while the agents avoided inter-agent collisions and the center of the agents tracked the optimal trajectory.




\bibliographystyle{IEEEtran}
\bibliography{IEEEabrv,IEEE_TSG_00018_2010}

\begin{appendices}

\section{} \label{AppendixA}

In this appendix an explanation is given on how Assumptions  \ref{as-bound-single} and \ref{as-bound-double} can be, respectively, satisfied in Theorem  \ref{theoremsingledis1} and \ref{theorem1}. We focus on the more involved case in Theorem \ref{theorem1} while a similar argument holds in the case in Theorem \ref{theoremsingledis1}. Here we show that Assumption \ref{as-bound-double} holds if the cost functions with identical Hessians satisfy certain conditions, referred as Condition ($\star$) for convenience, such that the boundedness of $\norm{x_i(t)-x_j(t)}_2$ and $\norm{v_i(t)-v_j(t)}_2$ for all $t$ guarantees the boundedness of  $\norm{\nabla f_j(x_j,t) - \nabla f_i(x_i,t) }_2$, $\norm{\frac{d }{dt}\nabla  f_j(x_j,t)-\frac{d }{dt}\nabla  f_i(x_i,t)}_2$ and $\norm{\frac{\partial^2  }{\partial t^2}\nabla  f_j(x_j,t)-\frac{\partial^2 }{\partial t^2}\nabla  f_i(x_i,t)}_2$ for all $t$. As a result, if Condition ($\star$) is satisfied, then Assumption 4.2 holds.

%
%
In particular, we will show that under Condition ($\star$), there always exists a finite $\bar{\phi}$ which can be determined at time $t=0$. With $\bar{\phi}$ determined, using the algorithm \eqref{udbldis1}, $\|x_i(t)-x_j(t)\|$ and $\|v_i(t)-v_j(t)\|, \forall t$ and $\forall i,j \in {\mathcal I},$ will remain bounded for all $t \geq 0$, which implies that Assumption \ref{as-bound-double} holds. We show the argument in four steps.

 \begin{enumerate}
 \item With identical Hessians, Assumption \ref{as-bound-double} holds if $\norm{\nabla f_j(x_j,t) - \nabla f_i(x_i,t) }_2$, $\norm{\frac{d }{dt}\nabla  f_j(x_j,t)-\frac{d }{dt}\nabla  f_i(x_i,t)}_2$ and $\norm{\frac{\partial^2  }{\partial t^2}\nabla  f_j(x_j,t)-\frac{\partial^2 }{\partial t^2} \nabla  f_i(x_i,t)}_2, \forall t$ and $\forall i,j \in {\mathcal I},$ are bounded. Assume that the boundedness of $\norm{x_i(t)-x_j(t)}_2$ and $\norm{v_i(t)-v_j(t)}_2, \forall t$ guarantees the boundedness of  $\norm{\nabla f_j(x_j,t) - \nabla f_i(x_i,t) }_2$, $\norm{\frac{d }{dt}\nabla  f_j(x_j,t)-\frac{d}{dt}\nabla  f_i(x_i,t)}_2$ and $\norm{\frac{\partial^2  }{\partial t^2}\nabla  f_j(x_j,t)-\frac{\partial^2  }{\partial t^2}\nabla  f_i(x_i,t)}_2, \ \forall t$ and $\forall i,j \in {\mathcal I}$. This is ensured by Condition ($\star$). Denote the upper bounds on $\norm{x_i(t)-x_j(t)}_2$ and $\norm{v_i(t)-v_j(t)}_2, \forall t$ and $\forall i,j \in {\mathcal I},$ as, respectively, $\beta_x$ and $\beta_v$.
It is easy to see that if there exist constant $\beta_x$ and $\beta_v$, there exists a constant $\bar{\phi}$, which in turn guarantees the existence of bounded $\bar{\beta}$, where $\bar{\beta}  >\frac{ (N-1)\bar{\phi}}{2}$.
%

\item Our proof will be completed if we can show that there exist constant $\beta_x$ and $\beta_v$. In the remaining, we will show that not only there exist constant $\beta_x$ and $\beta_v$, but also it is sufficient to determine these constants using the agents' initial states.
Two conservative $\beta_x$ and $\beta_v$ are selected using the initial states as
\begin{equation}\label{bxbv} \small
\begin{split}
\beta_x=&2\sqrt{\frac{m\lambda_\text{max}[P]}{N\lambda_\text{min}[P]}}\big(\max_{i} \norm{\sum_{j=1}^N(x_i(0) - x_j(0))}_\infty\\
&+\max_{i} \norm{\sum_{j=1}^N(v_i(0) -v_j(0))}_\infty\big)+\gamma\\
\beta_v=&2\sqrt{\frac{m\lambda_\text{max}[P]}{N\lambda_\text{min}[P]}}\big(\max_{i} \norm{\sum_{j=1}^N(x_i(0) - x_j(0))}_\infty\\&+\max_{i} \norm{\sum_{j=1}^N(v_i(0) -v_j(0))}_\infty\big)+\gamma, 
\end{split}
\end{equation}\normalsize where $\gamma$ is a positive constant, $\lambda_\text{min}[P]>0,$ and $\lambda_\text{max}[P]>0$ with $P$ defined in \eqref{lyapdbldis}.
Now, we will show that $\norm{x_i(t)-x_j(t)}_2 \leq \beta_x, \norm{v_i(t)-v_j(t)}_2 \leq \beta_v, \ \forall i,j \in {\mathcal I}$ and $\forall t$.

\item We know that for $W$ defined in \eqref{lyapdbldis} and for $\bar{\beta}$ selected based on the introduced constants $\beta_x$ and $\beta_v$, we have $\dot{W}_{|t=0} < 0$. 
We will use a contradiction approach to show that such $\bar{\beta}$ ensures $\dot{W} \leq 0, \forall t$. Assume that there exists a time $t=t_{\epsilon}$ at which $\dot{W}$ becomes positive, i.e, $\dot{W}_{|t=t_{\epsilon}} > 0$. 
By recalling the selected conservative $\bar{\beta}$, this is only possible if one or more of the constants $\beta_x,$ and $\beta_v$ do not exist. This means that there exist two agents $k$ and $l$ such that $\norm{x_k(t_{\epsilon})-x_l(t_{\epsilon})}_2> \beta_x$
or $\norm{v_k(t_{\epsilon})-v_l(t_{\epsilon})}_2 >\beta_v$.
Let us first suppose that $\norm{x_k(t_{\epsilon})-x_l(t_{\epsilon})}_2> \beta_x$. Note that that $\dot{W}(t) \leq 0, \ \forall t \in [0,t_{\epsilon}),$ which means that $W(t) \leq W(0), \ \forall t \in [0,t_{\epsilon})$. Using \eqref{lyapdbldis}, it is easy to see that $\forall t \in [0,t_{\epsilon})$ we have $\sqrt{\frac{\lambda_\text{max}[P]}{\lambda_\text{min}[P]}} \norm{\left( {\begin{array}{cc} e_X(0) \\ e_V(0) \\ \end{array} } \right)}_2 \geq \norm{\left( {\begin{array}{cc} e_X(t) \\ e_V(t) \\ \end{array} } \right)}_2.$
Now, using the graph connectivity and the properties of the norms, it is easy to show that
\begin{equation*}\small
\begin{split}
&\beta_x-\gamma= 2\sqrt{\frac{m\lambda_\text{max}[P]}{N\lambda_\text{min}[P]}}\big(\max_{i} \norm{\sum_{j=1}(x_i(0) - x_j(0))}_\infty\\
&+\max_{i} \norm{\sum_{j=1}(v_i(0) -v_j(0))}_\infty\big)=2 \sqrt{\frac{Nm\lambda_\text{max}[P]}{\lambda_\text{min}[P]}}(\norm{ e_X(0)}_\infty\\
&+\norm{ e_V(0)}_\infty)\geq 2 \sqrt{\frac{Nm\lambda_\text{max}[P]}{\lambda_\text{min}[P]}}\norm{\left( {\begin{array}{cc} e_X(0) \\ e_V(0) \\ \end{array} } \right)}_\infty 
\end{split}
\end{equation*} 
\begin{equation*}\small
\begin{split}
& \geq 2 \sqrt{\frac{\lambda_\text{max}[P]}{\lambda_\text{min}[P]}} \norm{\left( {\begin{array}{cc} e_X(0) \\ e_V(0) \\ \end{array} } \right)}_2 \geq 2 \norm{\left( {\begin{array}{cc} e_X(t) \\ e_V(t) \\ \end{array} } \right)}_2\\
&\geq 2 \norm{e_X(t)}_2\geq
2 \max_{i} \norm{x_i(t)-\frac{1}{N}\sum_{j=1}x_j(t)}_\infty\\&\geq \norm{x_k(t)-x_l(t)}, \forall k,l \in  {\mathcal I}, \text{and} \ \forall \ t \in [0,t_{\epsilon}).
\end{split}
\end{equation*} \normalsize
Now, under the assumption of $\norm{x_k(t_{\epsilon})-x_l(t_{\epsilon})}_2> \beta_x$ and using the selected $\beta_x$ in \eqref{bxbv}, we have $\norm{x_k(t_{\epsilon})\\-x_l(t_{\epsilon})}_2- \lim_{t \to t_{\epsilon}^-} \norm{x_k(t)-x_l(t)}_2> \gamma$. However, $\lim_{t \to t_{\epsilon}^-} \norm{x_k(t)-x_l(t)}_2 \neq \norm{x_k(t_{\epsilon})-x_l(t_{\epsilon})}_2$, contradicts with the continuity of the agents' positions as mentioned in Remark \ref{fillipov}.
Therefore, we have $\norm{x_i(t)-x_j(t)}_2 \leq \beta_x, \ \forall i,j \in {\mathcal I}$ and $\forall t$. The same argument can be made for showing $\norm{v_i(t)-v_j(t)}_2 \leq \beta_v, \ \forall i,j \in {\mathcal I}$ and $\forall t$, which is omitted here. We thus conclude that there exists no time $t=t_{\epsilon}$, where $\dot{W}_{|t=t_{\epsilon}} > 0$.
\end{enumerate}

For example, to satisfy Assumption \ref{as-bound-double} for the local cost function $f_i(x_i,t) = (a x_i + g_i(t))^2$ defined in Remark \ref{Remarkboundsingle}, it is only required to satisfy Condition ($\star$).  It is easy to see that  Condition ($\star$) boils down to the boundedness of $\norm{g_i(t)-g_j(t)}_2, \norm{\dot{g}_i(t)-\dot{g}_j(t)}_2$, and $\norm{\ddot{g}_i(t)-\ddot{g}_j(t)}_2, \forall t$ and $\forall i,j \in {\mathcal I}$. Hence the boundedness of $\norm{g_i(t)-g_j(t)}_2, \norm{\dot{g}_i(t)-\dot{g}_j(t)}_2,$ and  $\norm{\ddot{g}_i(t)-\ddot{g}_j(t)}_2, \forall t$ and $\forall i,j \in {\mathcal I},$  is sufficient to ensure that Assumption \ref{as-bound-double} holds.

A similar argument can be done for satisfying Assumption \ref{as-bound-single} in Theorem \ref{theoremsingledis1}. Here for cost functions with identical Hessians, Condition ($\star$) is such that the boundedness of $\norm{x_i(t)-x_j(t)}_2$ for all $t$ guarantees the boundedness of $\norm{\nabla f_j(x_j,t) - \nabla f_i(x_i,t)}_2$ and $\norm{\frac{\partial }{\partial t}\nabla  f_j(x_j,t)-\frac{\partial }{\partial t}\nabla  f_i(x_i,t)}_2, \forall i,j \in {\mathcal I}, $ for all $t$. As mentioned in Remark \ref{Remarkboundsingle}, for the local cost function $f_i(x_i,t) = (a x_i + g_i(t))^2$, Condition ($\star$) boils down to the boundedness of $\norm{g_i(t) -g_j(t)}_2$ and $\norm{\dot{g}_i(t) -\dot{g}_j(t)}_2, \forall t$ and $\forall i,j \in {\mathcal I}$. 
 Hence the boundedness of  $\norm{g_i(t) -g_j(t)}_2$ and $\norm{\dot{g}_i(t) -\dot{g}_j(t)}_2$ is sufficient to ensure that Assumption  \ref{as-bound-single} holds.

\section{} \label{AppendixbB}

In this appendix, we clarify how the boundedness of $\phi_i, \forall i  \in {\mathcal I}$ and $\forall t$, in Theorem \ref{theoremswarm}, can be satisfied. In particular, we show that for bounded $\norm{g_i(t)}_2$ and $\norm{\dot{g}_i(t)}_2, \forall t$ and $\forall i \in \mathcal{I}$, there exists a constant $\beta$ such that $\beta > \norm{\phi_i}_1, \ \forall i \in {\mathcal I}$ and $\forall t$. The constant $\beta$ can be determined at time $t=0$ using the agents' initial states and the upper bounds on $\norm{g_i(t)}_2$ and $\norm{\dot{g}_i(t)}_2, \forall t$ and $\forall i \in \mathcal{I}$.

Denote the upper bounds on $\norm{x_i(t)}_2, \norm{g_i(t)}_2$ and $\norm{\dot{g}_i(t)}_2, \forall t$ and $\forall i \in \mathcal{I},$ as, respectively,  $\beta_x, \bar{g}$ and $\bar{\dot{g}}$. It is easy to see that if there exist constant $\beta_x, \bar{g}$ and $\bar{\dot{g}}$, there exists a constant $\beta$, where $\beta > \norm{\phi_i}_1,  \forall i \in {\mathcal I}$ and  $\forall t$. We show the argument in three steps.

\begin{enumerate}

\item It is assumed that $\norm{g_i(t)}_2$ and $\norm{\dot{g}_i(t)}_2, \forall t$ and $\forall i \in \mathcal{I},$ are bounded. Hence, our proof will be completed if we can show that there exists a constant $\beta_x$. In the remaining, we will show that not only there exists a constant $\beta_x$, but also it is sufficient to determine this constant using the agents' initial states.
Define $\beta_x$ as
\begin{equation}\small \label{betax}
\beta_x=\frac{1}{N} \norm{\sum_{i=1}^N x_i(0)}_2+\frac{2}{\sigma}\bar{g}+(N-1)R+\gamma,
\end{equation} 
where $\gamma$ is a positive constant and $R$ is defined in Definition \ref{defswarm}.
We know that for $W$ defined in \eqref{Lyapswarm} and for $\beta$ selected based on the introduced constants $\beta_x, \bar{g}$ and $\bar{\dot{g}}$, we have $\dot{W}_{|t=0} < 0$. 
We will use a contradiction approach to show that such $\beta$ ensures $\dot{W} \leq 0, \forall t$. Assume that there exists a time $t=t_{\epsilon}$ at which $\dot{W}$ becomes positive, i.e, $\dot{W}_{|t=t_{\epsilon}} > 0$. 
By recalling the selected conservative $\beta$, this is only possible if there exists an agent $i$ such that $\norm{x_i(t_{\epsilon})}_2> \beta_x$.
Note that that $\dot{W}(t) \leq 0, \ \forall t \in [0,t_{\epsilon}),$ which means that $V_{ij}(t)$ is bounded, $\forall t \in [0,t_{\epsilon})$, which in turn implies that the agents remain connected. Hence, it is easy to see that $\forall t \in [0,t_{\epsilon})$ we have 
\begin{align} \label{bound-local-sw}
\norm{x_i(t)-x_j(t)}_2 <(N-1)R.
\end{align}

\item Using $W$ defined in \eqref{Wsingledis} and similar to Theorem \ref{theoremsingledis1}, we have $\dot{W}(t)< 0$ (no matter consensus is reached or not), which in turn implies that $\norm{\sum_{j=1}^N \nabla f_j(x_j,t)}_2$ is decreasing. Now, for $\nabla f_i(x_i,t)=\sigma x_i+g_i(t)$, defined in Theorem \ref{theoremswarm}, and using the properties of the norms, it is easy to show that
\begin{equation*}\small
\begin{split}
\sigma &\norm{\sum_{i=1}^N x_i(t)}_2-\norm{\sum_{i=1}^N g_i(t)}_2\leq\norm{\sum_{i=1}^N \nabla f_i(x_i(t),t)}_2\\
\leq &\norm{\sum_{i=1}^N \nabla f_i(x_i(0),0)}_2\leq \sigma\norm{\sum_{i=1}^N x_i(0)}_2+\norm{\sum_{i=1}^N g_i(0)}_2.
\end{split}
\end{equation*}
Using the upper bound $\bar{g}$, we have $\forall t$
\begin{equation} \label{bound-sw-x}
\norm{\sum_{i=1}^N x_i(t)}_2\leq \norm{\sum_{i=1}^N x_i(0)}_2+\frac{2N}{\sigma}\bar{g}.
\end{equation}
\item Now, using \eqref{bound-local-sw} and \eqref{bound-sw-x}, it is easy to see that 
\begin{equation}
\norm{x_i(t)}_2\leq\frac{1}{N} \norm{\sum_{i=1}^N x_i(0)}_2+\frac{2}{\sigma}\bar{g}+(N-1)R,
\end{equation}
$ \forall i \in \mathcal{I}$ and $\forall  t \in  [0,t_{\epsilon}).$ Now, under the assumption of $\norm{x_i(t_{\epsilon})}_2> \beta_x$ and using the selected $\beta_x$ in \eqref{betax}, we have $\norm{x_i(t_{\epsilon})}_2- \lim_{t \to t_{\epsilon}^-} \norm{x_i(t)}_2> \gamma$. However, $\norm{x_i(t_{\epsilon})}_2 \neq \lim_{t \to t_{\epsilon}^-} \norm{x_i(t)}_2$ contradicts with the continuity of the agents' positions as mentioned in Remark \ref{fillipov}.
Therefore, we have $\norm{x_i(t)}_2 \leq \beta_x, \ \forall i \in {\mathcal I}$ and $\forall t$. We thus conclude that there exists no time $t=t_{\epsilon}$, where $\dot{W}_{|t=t_{\epsilon}} > 0$.

\end{enumerate}

%

\end{appendices}
 
\begin{IEEEbiography}
[{\includegraphics[width=1in,height=1.25in,clip,keepaspectratio]{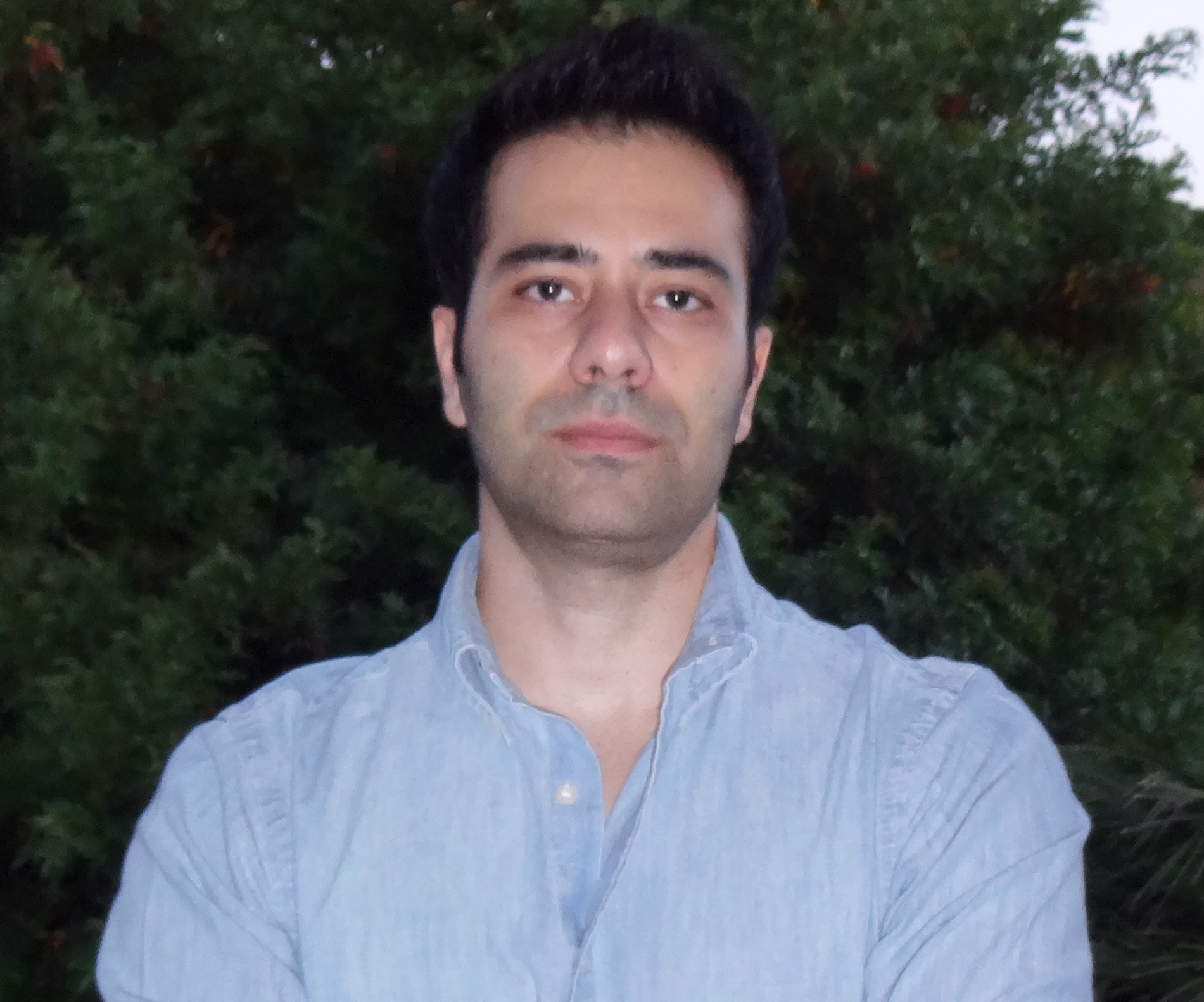}}]
{Salar Rahili}received the B.S. and the M.Sc. degrees in Electrical Engineering from Isfahan University of Technology, Isfahan, Iran, in 2009 and 2012, respectively. He is currently pursuing his Ph.D. degree in Electrical Engineering at the University of California, Riverside. His research focuses on distributed control of multi-agent systems, distributed optimization and game theory.
\end{IEEEbiography}
\begin{IEEEbiography}
[{\includegraphics[width=1in,height=1.25in,clip,keepaspectratio]{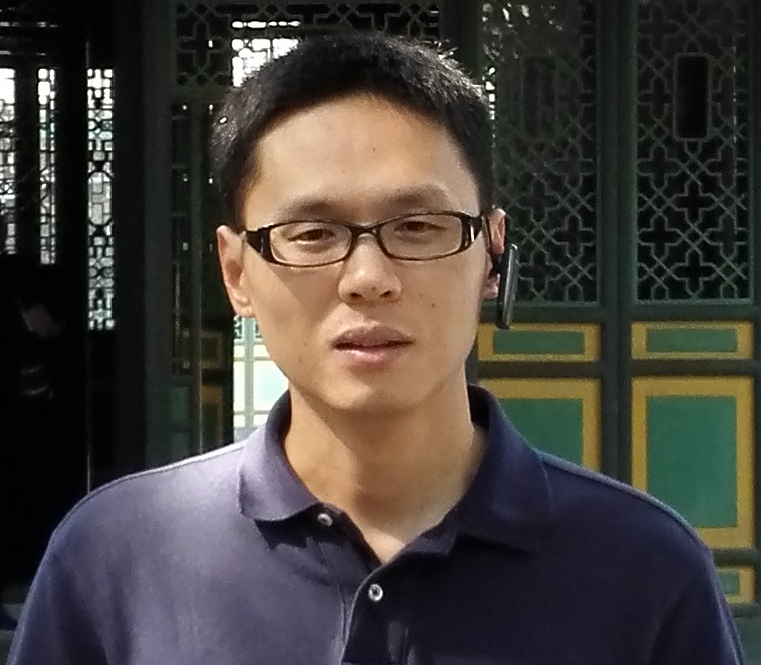}}]{Wei Ren}
is currently a Professor with the Department of Electrical and Computer Engineering, University of California, Riverside. He received the Ph.D. degree in Electrical Engineering from Brigham Young University, Provo, UT, in 2004. From 2004 to 2005, he was a Postdoctoral Research Associate with the Department of Aerospace Engineering, University of Maryland, College Park. He was an Assistant Professor (2005-2010) and an Associate Professor (2010-2011) with the Department of Electrical and Computer Engineering, Utah State University. His research focuses on distributed control of multi-agent systems and autonomous control of unmanned vehicles. Dr. Ren is an author of two books Distributed Coordination of Multi-agent Networks (Springer-Verlag, 2011) and Distributed Consensus in Multi-vehicle Cooperative Control (Springer-Verlag, 2008). He was a recipient of the National Science Foundation CAREER Award in 2008. He is currently an Associate Editor for Automatica, Systems and Control Letters, and IEEE Transactions on Control of Network Systems.
\end{IEEEbiography}

\end{document}